%% file: manuscript.tex
\documentclass[hidelinks,onefignum,onetabnum]{siamart250211}

\usepackage[left=3cm, right=3cm]{geometry}

\input{ex_shared}

\ifpdf
\hypersetup{
  pdftitle={A free boundary model for transport induced neurite growth},
  pdfauthor={G. Marino, J.-F. Pietschmann, and M. Winkler}
}
\fi

\begin{document}

\maketitle

\begin{abstract}
We introduce a free boundary model to study the effect of vesicle transport onto neurite growth. 
It consists of systems of drift-diffusion equations describing the evolution of the density of antero- and retrograde vesicles in each neurite coupled to reservoirs located at the soma and the growth cones of the neurites, respectively. 
The model allows for a change of neurite length as a function of the vesicle concentration in the growth cones.
After establishing existence and uniqueness for the time-dependent problem, we briefly comment on possible types of stationary solutions.
Finally, we provide numerical studies on biologically relevant scales using a finite volume scheme. We illustrate the capability of the model to reproduce cycles of extension and retraction.
\end{abstract}

\vskip-.5em

\section{Introduction}
Mature neurons are highly polarized cells featuring functionally distinct compartments, the axons and the dendrites. {Axons are ``cables'' that have the ability to transmit electrical signals to other neurons and can extend up to a length of one meter in humans. Dendrites form complex tree-like structures and act as recipients for axons of other neurons.} This polarity is established during the maturing proceess as initially, newborn neurons  feature several undifferentiated extensions of similar length called neurites that are highly dynamic \cite{Cooper2013,Hatanaka2012}.
Eventually, one of these neurites is selected to become the axon. {This is often called neurite outgrowth.}
{The understanding of this process is still incomplete, despite progress in characterizing the role of molecular mechanisms as well as influence of intra- and extracellular signaling molecules, see \cite{takanoNeuronalPolarization2015} for more details and further references. \\
In this work, we focus on a single aspect of this process, namely the fact that the }actual growth or shrinkage of neurites is due to the insertion or retraction of vesicles (i.e., circular structures composed of lipid membranes) at the outer tips  of the neurites (growth cones). 
The vesicles themselves are produced in the cell body (soma) and then form complexes with motor proteins that allow for active transport along microtubules. The direction of transport is determined by the type of motor protein: kinesin results in anterograde transport (into the growth cones) while dynein motors move vesicles retrogradely to the soma. Both kinesins and dyneins are present on vesicles during their transport along microtubules, but only one of them is usually active at any given time \cite{enclada_stable_2011, twelvetrees_dynamic_2016}, see Figure \ref{fig:DevelopingNeuron} for a sketch.
The actual increase of the surface area of the plasma membrane is then due to the insertion of vesicles into the growth cone (exocytosis). Retraction, on the other hand, is accompanied by the removal of membrane material from  the growth cone through endocytosis \cite{Pfenninger2009, pfenninger_regulation_2003, tojima_exocytic_2015}.
{Clearly, the (dis-)assembly of microtubules during growth and retraction is important, yet we neglect this effect in the present study in order to not further complicate the model and as we are primarily interested in the role of vesicle transport. Addition of microtubule dynamics is postponed to future work.}
\begin{figure}
	\centering
	\includegraphics[width=0.35\textwidth]{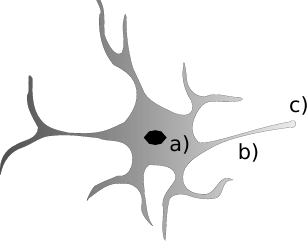}
	\caption{
	Sketch of a developing neuron. Here a) represents the cell nucleus/soma where vesicles are produced, b) a neurite and c) a  growth cone, i.e., the location where vesicles are inserted/removed into the cell membrane.
    }
	\label{fig:DevelopingNeuron}
\end{figure}

\subsection{Relation to existing work}
While there are different models for the underlying biochemical processes of selecting the neurite which eventually becomes the axon (see also \cite{Oliveri2022} for a recent review), mathematical models examining the role of vesicle transport in this process are relatively scarce. 
{On the other hand, there are several models for molecular motor based transport, also in axons, \cite{Smith2001ModelsOM,Friedman2005,Friedman2006,Newby2010,Bressloff2015_democracy}. All these models feature linear transport terms which do not take into account size exclusion or finite volume effects. 
Our starting point is a model with non-linear transport terms proposed in \cite{Bressloff2016_exclusion} which, again, focuses on transport in a grown axon.} In particular in \cite{Bressloff2016_exclusion}, a limited transport capacity inside the neurites is taken into account by size exclusion effects and antero- and retrogradely moving particles are modelled separately. {We will use this approach as a basis for the transport within the neurites in our model.}
%
In \cite{Humpert2021_transport} a similar approach is taken, yet on a microscopic particle level. Furthermore, \cite{Humpert2021_transport} extends the model by coupling two copies of it to pools representing the amount of vesicles present at the soma and growth cones, respectively. 
The aim of this paper is to introduce a macroscopic model in the spirit of both \cite{Humpert2021_transport,Bressloff2016_exclusion}, yet additionally allowing the length of the respective neurites to change. Different to \cite{Bressloff2016_exclusion} (see also \cite{Burger2010cross}), our model will have linear diffusion but non-linear transport terms. Such a model can also be justified as limit of a discrete lattice model, see \cite{KourbaneHoussene2018_exact,Bruna2022}. We are able to show that the solution stays within a given interval (usually taken to be $[0,1]$) so that the size exclusion property is preserved. Then, these equations which model transport inside the neurons are, as in \cite{Humpert2021_transport}, coupled to ordinary differential equations for the evolution of the vesicle concentration at soma and tip, respectively. One of the main novelties is then to add a mechanism which allows for growth or shrinkage of the neurites depending on how many vesicles are present in the growth cones. {Such free boundary models for neuron development have previously mostly been studies in the context of microtubule assembly, see \cite{McLean2004,diehl2014one,graham2006dynamics}. These models focus on a single neurite in which transport of microtubules is modelled again by a linear diffusion advection equation on a domain of varying length. This is then coupled to an ODE at one end of the domain accounting for the extension/retraction due to the microtubules. This coupling is sometimes performed via Dirichlet condition. Closer to our approach is the coupling through flux (Robin) type boundary condition as in \cite{Bressloff2015_democracy}. However, in this work, the authors only assume a linear relation for the boundary terms contrary to our study.}

\subsection{Contribution and outline}
We make the following contributions
\,

\noindent $\bullet$ Based on \cite{Humpert2021_transport,Bressloff2016_exclusion}, we introduce a macroscopic model for vesicle transport in developing neuron cells that includes multiple neurites, coupled with ODEs for the vesicle concentration in soma and growth cones. { We use a non-linear transport mechanism to include finite size effects extending paradigm used in most of the previous models.}

\noindent $\bullet$ We add a mechanism that allows for a change of neurite length depending on the respective vesicle concentration, which renders the model a free boundary problem.

\noindent $\bullet$ We rigorously prove  existence and uniqueness of solutions, including box constraints corresponding to size exclusion effects due to the finite volume of vesicles.


\noindent $\bullet$ We provide a finite volume discretization that preserves the box constraints.

\noindent $\bullet$ We perform a scaling of the model to biological reasonable regimes and then give some numerical experiments illustrating different behaviours of the model, in particular cycles of expansion and retraction as observed in experiments.

\medskip

The  paper is organised as follows. In Section~\ref{sect2} we present our model in detail. Section \ref{sect3} contains some preliminaries and is then devoted to  weak solutions, while Section~\ref{sect4} contains a brief discussion on (constant) stationary solutions. Section~\ref{sect5} provides a  finite volume scheme, a non-dimensionalization together with the introduction of biologically relevant  scales. Section \ref{sect6} is devoted to the numerical studies. Finally, Section~\ref{sect7} provides a brief conclusion and outlook.

\section{Mathematical model}
\label{sect2}

In this section we present a mathematical model for the growth process based on the principles stated in the introduction. For the reader's convenience, we will focus on the case of a two neurites connected to the soma, pointing out that the generalization to multiple neurites is straightforward. For $j=1,2$, the unknowns of our model read as follows:

        \noindent $\bullet$ $L_j(t)$  denotes the length of the respective neurite at time $t$;
        
        \noindent $\bullet$ $f_{+,j}(t,x)$ and $f_{-,j}(t,x)$  denote the density of motor-protein complexes in the neurite $j$ that move towards the growth cone (anterograde direction) and towards the soma (retrograde direction), respectively;

        
        \noindent $\bullet$ $\Lambda_{\text{som}}(t)$ is the amount of vesicles present in the soma at time $t$;
        
        \noindent $\bullet$ $\Lambda_j(t)$  is the amount of vesicles present in the tip of each neurite at time $t$.
    
\,
    
\begin{figure}[h]
	\centering
    \def\svgwidth{\textwidth}
	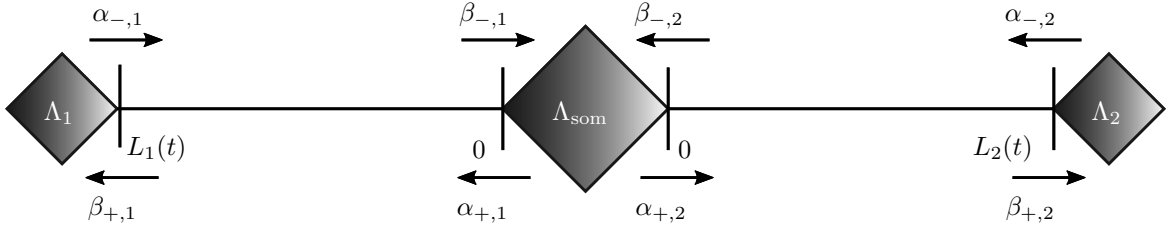
	\caption{{Sketch of the model neuron}: it consists of two neurites modelled by two intervals $(0, L_1(t))$ and $(0, L_2(t))$. The squares correspond to pools where vesicles can be stored. More precisely, the pool in the middle corresponds to the soma while the others stand for the corresponding growth cones. The interaction between neurites and pools is realised via boundary fluxes and the parameters governing their respective strength are displayed along with arrows of the transport direction. For an easy visualisation, $(0, L_1(t))$ is illustrated as a mirrored copy of $(0, L_2(t))$.}
	\label{fig:DiscreteModel}
\end{figure}

The complete model consists of equations governing the dynamics inside each neurite, coupled with ODEs at the soma and growth cones, respectively, as well as with equations accounting for the change of the neurites lengths, see Figure \ref{fig:DiscreteModel} for an illustration of the couplings. We will discuss each component separately.

\medskip
    
\noindent 1. \emph{Dynamics within the neurites.} Let $v_0>0$  be the velocity of vesicles as they move along neurites and let  $\rho_j = \rho_j(t,x) := f_{+,j}  + f_{-,j}$ be the total vesicle density,  $j= 1, 2$. We define the fluxes of antero- and retrogradely moving vesicle-motor complexes as
    \begin{equation}
    \begin{split}
    \label{eq:model}
    J_{+,j}&:= v_0\,f_{+,j}\,(1-  \rho_j)- D_T\,\partial_x f_{+,j} , \quad     J_{-,j}:= -v_0\,f_{-,j}\,(1-  \rho_j)- D_T\,\partial_x f_{-,j},
    \end{split}   
    \end{equation}
respectively, {where $D_T > 0$ is a positive diffusion constant.
Let us emphasize again that compared to earlier models, \cite{Smith2001ModelsOM,Friedman2005,Friedman2006,Newby2010,Bressloff2015_democracy}, we include a non-linear transport term to account for finite size effects.} We assume additionally that the complexes can (randomly, possibly via dissociation) change their direction with a given rate $\lambda \ge 0$. We obtain the following drift-diffusion-reaction equations, a copy of which holds true in each neurite separately,
    \begin{align} \label{eq:model_cont}
    \begin{aligned}
    \partial_t f_{+,j}&= -\partial_x J_{+,j} + \lambda\, (f_{-,j} - f_{+,j}), \\
    \partial_t f_{-,j}&= -\partial_x J_{-,j} + \lambda\, (f_{+,j} - f_{-,j}),
    \end{aligned}
\quad \text{in } (0, T) \times (0, L_j(t)).
    \end{align}
Here, $L_j(t)$ is the current length of the domain and $T>0$ is a fixed final time. 
{Note that the constants $v_0$, $D_T$ and $\lambda$ do not depent on $j$ as they are related to the characteristics of the transport of vesicle-motor protein complexes which are the same in every neurite.}
\medskip

\noindent 2. \emph{Coupling to soma and pools.} We assume that all neurites are connected to the soma at the point $x=0$. There, we have the following effects:

        \noindent $\bullet$ Retrograde vesicles leave the neurite and enter the soma with rate $\beta_{-,j}(\Lambda_{\text{som}})\, f_{-,j}$. {Here, the function $\beta_{-,j}$ allows for a control of incoming vesicles in terms of the available quantity in the soma. The intuition is that the soma is less likely to allow for incoming vesicles when it already contains a larger number of them.}
        
        \noindent $\bullet$ Anterograde vesicles can leave the soma and enter the lattice  
        with rate $\alpha_{+,j}(\Lambda_{\text{som}})\,g_{+,j}(f_{+,j}, f_{-,j})$ if there is enough space, i.e., if $\rho_j< 1$. {This is ensured by assuming that the non-negative function $g_{+,j}$ satisfies $g_{+,j}(f_{+,j}, f_{-,j}) = 0$ whenever $\rho_j = f_{+,j} + f_{-,j}=1$. The additional factor $\alpha_{+,j}(\Lambda_{\text{som}})$ reflects that the number of vesicles entering the neurite depends on the amount which is available within the soma. In particular, we ask for  $\alpha_{+,j}(0)=0.$}
        

At the point $x= L_j(t)$ the neurite is connected to its respective pool and we have:

        \noindent $\bullet$ Anterograde vesicles leave the lattice and enter the pool with rate $\beta_{+,j}(\Lambda_j)\,f_{+,j}$.
        
        \noindent $\bullet$ Retrograde particles move from the pool into the neurite, once again only if space in the domain is available, with rate $\alpha_{-,j}(\Lambda_j)\,g_{-,j}(f_{+,j},f_{-,j})$. 
    {Here, the functions $\beta_{+,j}$ and $\alpha_{-,j}$ serve the same purpose as $\beta_{-,j}$ and $\alpha_{+,j}$ previously, yet with pool instead of soma.}
 Figure \ref{fig:DiscreteModel} provides a sketch of this situation.    This behaviour is implemented by imposing the following flux boundary conditions (for each neurite):
%
    \begin{align}
     J_{+,j}(t,0) &= \alpha_{+,j}(\Lambda_{\text{som}}(t))\,g_{+,j}(\bbf_{j}(t,0)), \nonumber \\
    -J_{-,j}(t,0)  &= \beta_{-,j}(\Lambda_{\text{som}}(t))\,f_{-,j}(t,0), \nonumber \\
    J_{+,j}(t, L_j(t))- L'_j(t)\, f_{+,j}(t, L_j(t))  &= \beta_{+,j}(\Lambda_j(t))\,f_{+,j}(t, L_j(t)), \nonumber \\ 
    -J_{-,j}(t, L_j(t))+ L_j'(t)\, f_{-,j}(t, L_j(t)) &= \alpha_{-,j}(\Lambda_j(t))\,g_{-,j}{\bbf_{j}(t,L(t))}),     \label{boundary}
    \end{align}
     $j=1,2$,
for suitable functions $\alpha_{i,j},\,\beta_{i,j}$, and $g_{i,j}$, ${i= +,-}, j=1,2$,  whose properties will be specified later, and with the shortened notation $\bbf_j(\cdot \,, \cdot):= (f_{+,j}(\cdot \,, \cdot), f_{-,j}(\cdot \,, \cdot))$,  $j= 1, 2$.
{The additional terms on the left hand side of the boundary conditions at $L_j(t)$ in \eqref{boundary} account for the mass flux of vesicles that occurs when the length of the neurite changes. They are especially important in order to keep track of the total mass in the system, see also \cite{Portegies2010, Breden2021} for similar constructions.}
\medskip
    
\noindent 3. \emph{Dynamics of the free boundary.} We assume that the length of each neurite $L_j$ satisfies the following ordinary differential equation
    \begin{equation}
        \label{free-boundary}
        L'_j(t) = h_j(\Lambda_j(t), L_j(t)),  
    \end{equation}
 where $h_j$, $j= 1,2$, are smooth functions to be specified. We think of $h_j$ as functions that change sign at a critical concentration of $\Lambda_j$ (i.e., switch between growth or shrinkage), which may depend on the current length of the neurite itself (e.g. in order to stop shrinkage at a minimal length).

\medskip

\noindent 4. \emph{Dynamics in soma and growth cones.}  Finally, we  describe the change of number of vesicles in the
soma and the respective neurite growth cones, due to vesicles entering and leaving the pools. In addition, a production term is added at the soma, while for the growth cones we add terms that model the consumption or production of vesicles due to growth or shrinkage of the neurite, respectively. We obtain
     \begin{equation}
    \label{ode}
    \begin{aligned}
    \Lambda'_{\text{som}}(t) &=  \sum_{j=1,2}  \bigl( \beta_{-,j}(\Lambda_{\text{som}}(t))\,f_{-,j}(t,0) - \alpha_{+,j}(\Lambda_{\text{som}}(t))\,g_{+,j}(\bbf_{j}(t,0)) \bigr) + \gamma_{\text{prod}}(t), \\ 
     \Lambda'_{1}(t) &=  \beta_{+,1}(\Lambda_1(t))\,f_{+,1}(t, L_1(t))- \alpha_{-,1}(\Lambda_1(t))\, g_{-,1}(\bbf_{1}(t, L_1(t))) \\
     & \qquad  -\chi\, h_1(\Lambda_1(t), L_1(t)) , \\
    \Lambda'_{2}(t) &= \beta_{+,2}(\Lambda_2(t))\,f_{+,2}(t, L_2(t))- \alpha_{-,2}(\Lambda_2(t))\,g_{-,2}(\bbf_{2}(t, L_2(t))) \\
    & \qquad  -\chi \,h_2(\Lambda_2(t), L_2(t)), 
    \end{aligned}
    \end{equation}
    {where $\chi>0$ is a given parameter that has the units vesicles / length and describes the amount of vesicles needed for one unit of neurite length while $\gamma_{\rm prod}$ accounts for the amount of vesicles that are produced in the soma.}

\begin{remark} Note that, except for the influence of the growth term $\gamma_{\text{prod}}(t)$, the total mass is preserved. 
%
It is defined by the following quantity
\begin{equation*}
m(t) = \sum_{j=1,2}\left(\int_0^{L_j(t)} \rho_j(t,x)\dx + \Lambda_j(t) + \chi\,L_j(t)\right) + \Lambda_{\text{som}}(t),
\end{equation*}
where we emphasize that also the material of which the neurites are made of need to be taken into account which is done via the terms $\chi\,L_j(t)$.
Then, a formal calculation yields the following equation of the evolution of the total mass
\begin{equation*}
m(t) = m(0) + \int_0^t\gamma_{\text{prod}}(s)\,\mathrm ds.
\end{equation*}

\end{remark}

\color{black}

\section{Existence of weak solutions}
\label{sect3}

The aim of this section is to provide existence of a unique weak solution to the model \eqref{eq:model}--\eqref{ode}. Let us first give some preliminaries.

\subsection{Preliminaries}
Let $L>0$ and let $1 \le p< \infty$. We denote by $L^p(0,L)$ and $W^{1,p}(0,L)$ the usual Lebesgue and Sobolev spaces.
For $p= 2$ we  write $H^1(\Om)$ instead of $W^{1,2}(\Om)$. Furthermore,  $H^1(\Om)'$ is the dual space of $H^1(\Om)$.


It is well known (see e.g. \cite{Adams}) that there exists a unique linear, continuous map $\Gamma\colon W^{1,p}(\Om)  \to \R$ known as the trace map such that $\Gamma(u)= u(0)$ for all $u \in W^{1,p}(\Om) \cap C([0,L])$. In addition, let us recall the following trace estimate \cite[Theorem 1.6.6]{BrennerScott2008}      \begin{equation}
\label{eq:trace_theorem}    
    |u(0)| \le C_{\text{e}}\,\|u\|_{L^2(\Om)}^{1/2}\,\|u\|_{H^1(\Om)}^{1/2}.
        \end{equation}
%
%
%
Let $T>0$ and let $(B, \|\cdot\|_B)$ be a Banach space. For every $1 \le r\le \infty$ we denote by $L^r(0,T; B)$ the Bochner space of all measurable functions $u\colon [0,T] \to B$ with finite norm
    $$
    \|u\|_{L^r(0,T; B)}:= 
    \begin{cases}
        \displaystyle \Big(\int_0^T \|u(t)\|_B^r \; \dt\Big)^{1/r}, &\text{if}\ r<\infty,\\
        \displaystyle \esssup_{0 \le t \le T} \|u(t)\|_B, &\text{if}\ r=\infty.
    \end{cases}
    $$
Finally,  $C([0, T]; B)$ contains all continuous functions $u\colon [0, T] \to B$ such that
    \[
    \|u\|_{C([0, T]; B)}:= \max_{0 \le t \le T} \|u(t)\|_B< \infty.
    \]
We refer to \cite{Evans} as a reference for Bochner spaces. 
For every $a \in \R$ we set $a^{\pm}:= \max\{\pm a, 0\}$ and for $u \in W^{1,p}(\Om)$ we define $u^\pm(\cdot):= u(\cdot)^\pm$ and will use the fact that $u^\pm \in W^{1,p}(0,L)$.




\subsection{Transformation for a fixed reference domain}
{Before we give our definition of weak solutions, state the necessary assumptions and our main theorem, we transform \eqref{weak2} into an equivalent system set on a fixed reference domain. This facilitates the proofs and also the spaces that we need to work in.} To this end, we make the following change of variables
    \[
    y= y(t, x)=: \frac{x}{L(t)} \quad \longleftrightarrow \quad x=  L(t) y.
    \]
Then we define the  functions
    $   \bar f_i(t, y)= f_i(t, x)= f_i(t, L(t) y)$, 
and observe that
    \begin{equation}
        \label{chang-var}
        \begin{split}
        \partial_x f_i=  \frac{1}{L(t)}\, \partial_y \bar f_i,  \quad \partial_t f_i= \partial_t \bar f_i- L'(t)\,y\, \partial_x f_i= \partial_t \bar f_i- \frac{L'(t)}{L(t)}\, y\, \partial_y \bar f_i, \quad \dx= L(t) \dy.
        \end{split}
    \end{equation}
Using \eqref{chang-var}, taking into account that, by the product rule, $y\, \partial_y \bar f_+= \partial_y(y\, \bar f_+)- \bar f_+$ and rearranging, the first equation of \eqref{eq:model_cont} reads as
    \begin{align}
        \partial_t \bar f_+&= -\frac{1}{L^2(t)}\, \partial_y \big(L(t)\, v_0\, \bar f_+\,(1- \bar \rho)- D_T\, \partial_y \bar f_+ - L'(t)\,L(t)\, y\, \bar f_+\big) - \frac{L'(t)}{L(t)}\, \bar f_++ \lambda\,(\bar f_- - \bar f_+) \nonumber \\
        &= -\frac{1}{L^2(t)}\, \partial_y \bar J_+- \frac{L'(t)}{L(t)}\, \bar f_++ \lambda\,(\bar f_- - \bar f_+), 
        \label{eq:discr1}
    \end{align}
    and with similar arguments
    \begin{equation*}
        \partial_t \bar f_- = -\frac1{L^2(t)}\,\partial_y \bar J_- - \frac{L'(t)}{L(t)}\, \bar f_- +\lambda\,(\bar f_+ - \bar f_-),
    \end{equation*}
    where the fluxes are defined by
    \begin{align*}
    \bar J_+(t,y) &= \phantom{-}L(t)\,v_0\,\bar f_+(t,y)\,(1- \bar \rho(t,y))- D_T\, \partial_y \bar f_+(t,y) - L'(t)\,L(t)\,y\,\bar f_+(t,y), \\
    \bar J_-(t,y) &= -L(t)\,v_0\,\bar f_-(t,y)\,(1- \bar \rho(t,y))- D_T\, \partial_y \bar f_-(t,y) - L'(t)\,L(t)\,y\,\bar f_-(t,y).
    \end{align*}
Note that the fluxes $J_+$ and $\bar J_+$ are related to each other via
    \begin{align*}
        & J_+(t, y\,L(t))- L'(t)\, y\, f_+(t, y\,L(t))  \\
        &\qquad= v_0 \,f_+(t, y\,L(t))\, (1- \rho(t, y\,L(t)))- D_T\, \partial_x f_+(t, y\,L(t))- L'(t)\,y\,f_+(t, y\,L(t))  \\
        &\qquad= v_0 \,\bar f_+(t,y)\,(1- \bar \rho(t, y))- \frac{D_T}{L(t)}\, \partial_y \bar f_+(t,y)- L'(t) \,y\,\bar f_+(t,y) \\
        &\qquad=\frac{1}{L(t)}\, \big(L(t)\, v_0\, \bar f_+(t,y)\,(1- \bar \rho(t,y))
        - D_T\, \partial_y \bar f_+(t,y)- L'(t)\,y\,\bar f_+(t,y)\big) \\
        &\qquad= \frac{1}{L(t)} \,\bar J_+(t,y),
    \end{align*}
    and a similar relation can be deduced for $J_-$ and $\bar J_-$. The boundary conditions \eqref{boundary} in the reference configuration then read
    \begin{equation}
        \label{boundary_ref_domain}
        \begin{aligned}
            \bar J_+(t,0) &= L(t)\,\alpha_+(\Lambda_{\text{som}}(t))\,g_+(\bar \bbf(t,0)),  \\
            \bar J_+(t,1) &= L(t)\,\beta_+(\Lambda(t))\,\bar f_+(t, 1), \\
            -\bar J_-(t,0) &= L(t)\,\alpha_-(\Lambda_{\text{som}}(t))\,g_-(\bar \bbf(t,0)),  \\
            -\bar J_-(t,1) &= L(t)\,\beta_-(\Lambda(t))\,\bar f_+(t, 1).            
        \end{aligned}
    \end{equation}
%
%
The ODE \eqref{free-boundary} remains unchanged, while for  \eqref{ode} quantities are evaluated at $y=1$ instead of $x=L_j(t)$, which results in 
     \begin{equation}
    \label{ode_transformed}
    \begin{aligned}
    \Lambda'_{\text{som}}(t) &=  \sum_{j=1,2}  \bigl( \beta_{-,j}(\Lambda_{\text{som}}(t))\,\bar f_{-,j}(t,0) - \alpha_{+,j}(\Lambda_{\text{som}}(t))\,g_{+,j}(\bar \bbf_{j}(t,0)) \bigr) + \gamma_{\text{prod}}(t), \\ 
     \Lambda'_{1}(t) &=  \beta_{+,1}(\Lambda_1(t))\,\bar f_{+,1}(t, 1)- \alpha_{-,1}(\Lambda_1(t))\, g_{-,1}(\bar \bbf_{1}(t, 1))
     -\chi\, h_1(\Lambda_1(t), L_1(t)) , \\
    \Lambda'_{2}(t) &= \beta_{+,2}(\Lambda_2(t))\,\bar f_{+,2}(t, 1)- \alpha_{-,2}(\Lambda_2(t))\,g_{-,2}(\bar \bbf_{2}(t, 1))
     -\chi\, h_2(\Lambda_2(t), L_2(t)), 
    \end{aligned}
    \end{equation}
    for $t\in (0,T)$.
%
 %
\color{black}

\subsection{Notion of weak solution and existence result}

We now define the notion of weak solution to our problem. Whenever not differently specified, we assume $i \in \{+,-\}$ as well as $j \in \{1,2\}$, $k \in \{1,2, \text{som}\}$, while $C>0$ denotes a constant that may change from line to line, but always depends only on the data. 
From now on we always write $f_{i,j}$ instead of $\bar f_{i,j}$ as we always work with the equations on the reference interval.



\begin{definition}
\label{def-1}
We say that $(\bbf_1, \bbf_2, \Lambda_{\emph{som}}, \Lambda_1, \Lambda_2, L_1, L_2)$ is a weak solution to \eqref{eq:discr1}--\eqref{ode_transformed}, \eqref{free-boundary} if 
    
       \noindent (a) $0\le f_{i,j} \le 1$ as well as $\rho_j:= f_{+,j}+ f_{-,j} \le 1$ a.e. in\ $(0,T)\times(0,1)$;
       
       \noindent (b) $  f_{i,j} \in L^2(0,T; H^1(0, 1))$ with $\partial_t f_{i,j} \in L^2(0,T; H^1(0, 1)')$;
       
       \noindent (c) $\bbf_j$ solves \eqref{eq:discr1}--\eqref{boundary_ref_domain} in the following weak sense
    \begin{equation}\label{weak2}
    \begin{aligned}
        \int_0^1 \partial_t f_+\, \varphi_+\,\mathrm dy 
       &=  \int_0^1 \frac{1}{L^2(t)} \Big[L(t)\,v_0\,f_+\,(1- \rho) - D_T\,\partial_y f_+ - L'(t)\,L(t)\,y\,f_+\Big]\,\partial_y \varphi_+ \\
       &\quad + \Big(\lambda(f_- - f_+) -\frac{L'(t)}{L(t)}\,f_+\Big) \,\varphi_+ \,\mathrm dy \\
        & \quad + \frac{1}{L(t)}\,\alpha_+(\Lambda_{\rm som}(t))\,g_+(\bbf(t,0))- \frac{1}{L(t)} \,\beta_+(\Lambda(t))\, f_+(t, 1) \,\varphi_+(1),  \\
        \int_0^1 \partial_t f_-\,\varphi_-\, \mathrm dy 
       &=  \int_0^1 \frac{1}{L^2(t)} \Big[-L(t)\,v_0\,f_-\,(1- \rho)- D_T\, \partial_y f_- - L'(t)\,L(t)\,y\, f_-\Big]\,\partial_y \varphi_- \\
       &\quad + \Big(\lambda(f_+ - f_-)-\frac{L'(t)}{L(t)} f_-\Big)\, \varphi_- \,\mathrm dy  \\ 
       & \quad -\frac1{L(t)}\,\beta_-(\Lambda_{\rm som}(t))\,f_-(t,0)\,\varphi_-(0) +  \frac{1}{L(t)} \,\alpha_-(\Lambda(t))\,g_-(\bbf(t,1))\,\varphi_-(1), 
       \end{aligned}
    \end{equation}
    for every $\varphi_+, \varphi_- \in  H^1(0, 1)$ and almost all $t\in(0,T)$.

    
  \noindent      (d) $L_j(0)= L_j^0$, $\Lambda_k(0)= \Lambda_k^0$, and $\bbf_j(0, y)= \bbf^0_j(y)$ for almost all $y\in(0,1)$, for suitable $L_j^0, \Lambda_k^0$, and $f_{i,j}^0$;
        
 \noindent       (e)
        $\Lambda_k \in {C^1([0, T])}$ solves \eqref{ode_transformed};
        
 \noindent       (f) $L_j \in {C^1([0, T])}$ solves \eqref{free-boundary}.
    
\end{definition}
\vskip1em
We next state the assumptions on the data and non-linearities which read as follows\\
%

    
        \noindent (H$_0$)  
        $\Lambda_k^0>0$ and  $L_j^0 \ge L_{\text{min}, j}$, where  $L_{\text{min}, j}>0$ is given.

         \noindent (H$_1$) 
        For $f_{+,j}^0, f_{-,j}^0 \in L^2(0,1)$ it holds $ f_{+,j}^0, f_{-,j}^0 \ge 0$ and $0 \le \rho_j^0 \le 1$ a.e. in $(0,1)$, where $\rho_j^0:= f_{+,j}^0+ f_{-,j}^0$.
        
        \noindent (H$_2$)
        The nonlinearities $g_{i,j}\colon \R^2 \to \R_+$, are   Lipschitz continuous and such that $g_{i,j}(s,t)= 0$ whenever $s+t = 1$ as well as $g_{-,j}(s,0) = g_{+,j}(0,s) = 0$ for all $0 \le s \le 1$.
        
        \noindent (H$_3$)  The functions $h_j\colon \R_+ \times [ L_{\text{min}, j},+\infty)  \to \R$ are such that
        
                \noindent \, \, (i) {there exist non-negative functions $K_{h_j} \in L^\infty((0,\infty))$ such that 
                $$
                |h_j(t,a) - h_j(t,b)| \le K_{h_j}(t)|a-b|,
                $$
                for all $a,\,b \in [ L_{\text{min}, j},+\infty)$;}
                
                \noindent \, \, (ii) 
            %
       $h_j(s, L_{\text{min}, j})\ge 0$ for every $s\ge 0$.
        
        

        \noindent (H$_4$) {The functions $\alpha_{i,j}\colon \R_+ \to \R_+$ are increasing and  Lipschitz continuous. Moreover, $\alpha_{-,j}(t) \ge 0$ for all $t>0$ and $\alpha_{-,j}(0)= 0$.}
        
        \noindent (H$_5$) The functions $\beta_{i,j}\colon \R_+ \to \R_+$ are nonnegative and  Lipschitz continuous. Moreover, there exists $\Lambda_{j, \text{max}}>0$ such that $\beta_{+,j}(\Lambda_{j, \text{max}})= 0$.

        %
        \color{black}
        %
        
        \noindent (H$_6$) The parameters satisfy $  v_0, D_T, \chi >0$ and $ \lambda \ge 0$.

        \noindent (H$_7$) The function $\gamma_{\text{prod}} \colon \R_+ \to \R_+$ is such that $\lim_{t\to\infty} \gamma_{\text{prod}}(t) =0$.


\begin{remark}[Interpretation of the assumptions]
{Let us briefly discuss the meaning of the assumptions in terms of our application. Assumption (H$_0$) states that we start with a predefined number of neurites with length above a fixed minimal length and that all pools as well as the soma are non-empty. (H$_1$) is necessary as $f_{+,j}^0, f_{-,j}^0$ model densities and must therefore be non-negative and as we assume that there is a maximal density (due to the limited space in the neurites) normalized to $1$. In (H$_2$), the regularity is needed for the analysis and only a mild restriction. The remaining requirements are necessary to ensure that all densities remain between $0$ and $1$. (H$_3$) ensures that there is a lower bound for the length of the neurites meaning that neurites cannot vanish as it is observed in practice. (H$_4$)  ensures that vesicles can only enter neurites if there are some available in growth cone or soma, respectively, while (H$_5$) allows the pools to decrease the rate of entering vesicles when they become too crowded. Finally, (H$_6$) states that diffusion, transport and reaction effects are all present at all times (yet with possible different strengths) and (H$_7$) finally postulates the production of vesicles within the soma.}
{We point out that assumption (H$_7$) is only needed to guarantee existence of stationary solutions.}
    
\end{remark}

\color{black}

\medskip

Then we can state our  existence result.

\begin{theorem}
\label{thm:existence}
Let the assumptions \emph{(H$_0$)}--\emph{(H$_6$)} hold.
Then, for every $T>0$ there exists a unique weak solution 
$(\bbf_1, \bbf_2, \Lambda_{\emph{som}}, \Lambda_1, \Lambda_2, L_1, L_2)$ to \eqref{eq:discr1}--\eqref{ode_transformed}, \eqref{free-boundary} in the sense of Definition \ref{def-1}.
\end{theorem}






\subsection{Proof of Theorem \ref{thm:existence}}
The proof of Theorem \ref{thm:existence} is based on a fixed point argument applied to an operator obtained by concatenating linearized versions of \eqref{weak2}, \eqref{free-boundary}, and \eqref{ode_transformed}.
Let us briefly sketch our strategy before we provide the corresponding rigorous results. 
We work in the Banach space 
    \[
    X= \prod_{j=1,2} L^2(0,T; H^1(0, 1))^2 
    \]
endowed with the norm
    \[
    \|(\bbf_1, \bbf_2
    )\|^2_X= \sum_{j=1,2} \sum_{i=+,-} \|f_{i,j}\|^2_{L^2(0,T; H^1(0, 1))}
    .
    \]
{Given $(\widehat \bbf_1, \widehat \bbf_2
) \in X$, }  let  $\boldsymbol{\Lambda}:=(\Lambda_{\text{som}}, \Lambda_1, \Lambda_2) \in C^1([0,T])^3$
be the unique solution to {the ODE system}
    \begin{equation}
    \label{linear-ode}
    \begin{aligned}
    \Lambda_{\text{som}}'(t) &=  \sum_{j=1,2}  \bigl( \beta_{-,j}(\Lambda_{\text{som}}(t)) \widehat f_{-,j}(t,0) - \alpha_{+,j}(\Lambda_{\text{som}}(t))\, g_{+,j}(\widehat \bbf_{j}(t,0)) \bigr) + \gamma_{\text{prod}}(t), \\ 
    \Lambda_{1}'(t) &=   \beta_{+,1}(\Lambda_1(t))\,\widehat f_{+,1}(t, 1)- \alpha_{-,1}(\Lambda_1(t))\,g_{-,1}(\widehat \bbf_{1}(t, 1))  - \chi\, h_1(\Lambda_1(t), L_1(t)), \\
    \Lambda_{2}'(t) &= \beta_{+,2}(\Lambda_2(t))\,\widehat f_{+,2}(t, 1)- \alpha_{-,2}(\Lambda_2(t))\,g_{-,2}(\widehat \bbf_{2}(t, 1))  - \chi\,h_2(\Lambda_2(t), L_2(t)). 
    \end{aligned}
    \end{equation}
We denote the mapping $(\widehat \bbf_1, \widehat \bbf_2
) \mapsto \boldsymbol{\Lambda}$ by $\mathcal{B}_1$.     
This $\boldsymbol{\Lambda}$ is now substituted into  \eqref{free-boundary}, that is, we are looking for the unique solution $\boldsymbol{L}=(L_1, L_2) \in C^1([0,T])^2$ to {the  ODE problems}
            \begin{equation}
              \label{free-bdry2}
                L_j'(t)= h_j(\Lambda_j(t), L_j(t)),
            \end{equation}
and the corresponding solution operator is denoted by $\mathcal{B}_2$. 
Finally, these solutions $\boldsymbol{\Lambda}$ and  $\boldsymbol{L}$  are substituted into  \eqref{weak2}, and  we look for the unique solution 
 $\bbf_j \in L^2(0, T; H^1(0, 1))^2$, with  $\partial_t \bbf_j \in L^2(0, T; H^1(0, 1)^{\prime})^2$,  to { the (still non-linear) PDE problem}
    \begin{align}
        \label{weak2.1}
           & \int_0^{1} \partial_t f_{+,j} \,\varphi_+ \dy = \int_0^{1} \frac{1}{L(t)^2}\left[L(t)\,v_0\,f_{+,j}\,(1-  \rho_j)- D_T\,\partial_y f_{+,j} - L'(t)\,L(t)\,y\,f_+ \right] \partial_y \varphi_+ \\
           &\quad +\left(\lambda\, (f_{-,j}- f_{+,j}) - \frac{L'(t)}{L(t)}\,f_{+,k}\right)\, \varphi_+ \dy 
           \nonumber \\
           & + \frac{1}{L(t)}\left[\alpha_{+,j}(\Lambda_{\text{som}})\,g_{+,j}(\bbf_{j}(t, 0))\,\varphi_+(0) - \beta_{+,j}(\Lambda_j)\,f_{+,j}(t, 1)\,\varphi_+(1)\right] ,\nonumber    
    \end{align}
    \begin{align}
     \label{weak2.2} 
            & \int_0^{ 1} \partial_t f_{-,j}\,\varphi_- \dy =  \int_0^{ 1} - \frac{1}{L(t)^2}\left[L(t)\,v_0\,f_{-,j}(1-  \rho_j)+ D_T\,\partial_y f_{-,j} - L'(t)\,L(t)\,y\,f_-\right]\, \partial_y \varphi_-
            \\
            &\quad +\left(\lambda\, (f_{+,j}- f_{-,j}) - \frac{L'(t)}{L(t)}\,f_{-,k}\right)\, \varphi_- \dy \nonumber \\
            &-\frac{1}{L(t)}\left[\beta_{-,j}(\Lambda_{\text{som}}(t)) \,f_{-,j}(t, 0)\,\varphi_-(0) 
            - \alpha_{-,j}(\Lambda_j)\,g_{-,j}(\bbf_{j}(t, 1))\,\varphi_-(1)\right],\nonumber
    \end{align}
for every $\varphi_+, \varphi_- \in  H^1(0, 1)$.
%
 %
%
	%
	%
 We call the resulting solution operator $(\boldsymbol{\Lambda}, \boldsymbol{L}) \mapsto (\bbf_1, \bbf_2
 )$  $\mathcal{B}_3$.
	%
Then, given an appropriate subset $\mathcal{K} \subset X$, we define the (fixed point) operator $\mathcal{B} \colon \mathcal{K} \to X $ as
	$$
	\mathcal{B}(\widehat \bbf_1, \widehat \bbf_2
 )= \mathcal{B}_3 \bigl(\mathcal{B}_1(\widehat \bbf_1, \widehat \bbf_2
 ), \mathcal{B}_2(\mathcal{B}_1(\widehat \bbf_1, \widehat \bbf_2
 )\bigr)= (\bbf_1, \bbf_2
 ).
	$$
We show that $\mathcal{B}$ is self-mapping and contractive, so that existence is a consequence of the Banach's fixed point theorem.

%
Let us begin with system \eqref{linear-ode}.
\begin{lemma} 
\label{lemma:lambda}

Let $(\widehat \bbf_1, \widehat \bbf_2
) \in X$,
then, there exists a unique  $\boldsymbol{\Lambda}= (\Lambda_{\emph{som}}, \Lambda_1, \Lambda_2) \in C^1([0, T])^3$ that solves \eqref{linear-ode} with initial conditions
    \begin{equation}
        \label{in-cond-Lambda}
        \Lambda_k(0)= \Lambda_k^{0}, \quad k= \emph{som}, 1, 2.
    \end{equation} 
\end{lemma}

\begin{proof}

This result is an application of the Cauchy-Lipschitz theorem, since the right-hand sides of \eqref{linear-ode} are Lipschitz continuous with respect to $\Lambda_k$ thanks to hypotheses (H$_4$) and (H$_5$).
\end{proof}

\begin{lemma}
\label{lemma:B1}
Let  $\mathcal{B}_1\colon X \to C([0, T])^3$ be the operator that maps $(\widehat \bbf_1, \widehat \bbf_2
) \in X$ to the solution $\boldsymbol{\Lambda}$ to \eqref{linear-ode}. Then, $\mathcal{B}_1$ is Lipschitz continuous.
%
\end{lemma}

\begin{proof}

From Lemma \ref{lemma:lambda}, $\mathcal{B}_1$  is well defined. Let now $(\widehat \bbf_1^{(1)},\widehat \bbf_2^{(1)}
),$ 
  $(\widehat \bbf_1^{(2)}, \widehat \bbf_2^{(2)}
  ) \in X$ and let $\boldsymbol{\Lambda}^{(1)}=: \mathcal{B}_1(\widehat \bbf_1^{(1)}, \widehat \bbf_2^{(1)}
  ) $ and $\boldsymbol{\Lambda}^{(2)}=: \mathcal{B}_1(\widehat \bbf_1^{(2)}, \widehat \bbf_2^{(2)}
  )$ be solutions to \eqref{linear-ode} satisfying the same initial condition \eqref{in-cond-Lambda}. We fix $t \in [0, T]$ and  consider 
    \begin{equation*}
        \begin{split}
        (\Lambda_{j}^{(a)})'(t)&=   \beta_{+,j}(\Lambda_j^{(a)}(t)) \,\widehat f_{+,j}^{(a)}(t, 1)- \alpha_{-,j}(\Lambda_j^{(a)}(t)) \,g_{-,j}(\widehat \bbf_{j}^{(a)}(t, 1)) \\
        & \qquad - \chi \, h_j(\Lambda_j^{(a)}(t), 1),
        \end{split}
            \end{equation*}
$a=1,2$. Taking the difference of the two equations, setting $\delta \Lambda_j:= \Lambda_j^{(1)}- \Lambda_j^{(2)}$, exploiting hypotheses (H$_2$)--(H$_5$) and summarizing the constants give
    \begin{equation*}
    \begin{split}
        & |(\delta \Lambda_j)'(t)| \le C
        |\delta \Lambda_j(t)| + C \l(|(\widehat f_{+,j}^{(1)}- \widehat f_{+,j}^{(2)})(t, 1)|+ |(\widehat f_{-,j}^{(1)}- \widehat f_{-,j}^{(2)})(t, 1)| \r),
    \end{split}
    \end{equation*}
while the trace inequality  \eqref{eq:trace_theorem} and a Gronwall argument imply 
    \begin{align*}
    |\Lambda_j^{(1)}(t)- \Lambda_j^{(2)}(t)| & \le C\, \|\widehat{\boldsymbol{f}}_j^{(1)}- \widehat{\boldsymbol{f}}_j^{(2)} \|_{L^2(0, T; H^1(0, 1))^2} .
    \end{align*}
 %
A similar argument holds for the  equation for $\Lambda_{\text{som}}$, and we eventually have
    \begin{equation}
    \label{lip-1}
    \|\boldsymbol{\Lambda}^{(1)}- \boldsymbol{\Lambda}^{(2)}\|_{C([0, T])^3} \le C \,\|(\widehat \bbf_1^{(1)}, \widehat \bbf_2^{(1)}
    )- (\widehat \bbf_1^{(2)}, \widehat \bbf_2^{(2)}
    )\|_X.
    \end{equation}
\end{proof}


We next show the following existence result for equation \eqref{free-bdry2}.

\begin{lemma}
\label{lemma:L}

Let $\boldsymbol{\Lambda} \in C^1([0, T])^3$ be the unique solution to \eqref{linear-ode}. Then, there exists a unique
$\boldsymbol{L}= (L_1, L_2) \in  C^1([0, T])^2$ that solves \eqref{free-bdry2} with initial condition $   L_j(0)= L_j^{(0)}$.
Furthermore, $\text{for all } t \in (0, T)$  it holds
    \begin{align}
    & L_{\emph{min}, j}  \le L_j(t) \le L_j^{(0)}+ T \,\|h_j\|_{L^{\infty}(\R^2)}, \label{dis-L} \\
      &  \frac{|L_j'(t)|}{L_j(t)}  \le \frac{\|h_j\|_{L^{\infty}(\R^2)}}{L_{\emph{min}, j}}. \label{rem:L}
\end{align}

\end{lemma}

\begin{proof}

The existence and uniqueness follow as before.
{The lower bound in \eqref{dis-L} can be deduced by applying Theorem \ref{thm:invariant_deimling} (\cite[Theorem 5.1]{Deimling}) in the appendix with $X=\R$, $\Omega = [L_{\text{min}, j},\infty)$ and $f=h_j$. Assumption (A1) in the theorem is satisfied as, due to (H$_3)$, the choice $\omega = K_{h_j}$ fulfils the requirements. For (A2), we note that the unit outward normal of the set $[L_{\text{min}, j},\infty)$ at $L_{\text{min},j}$ is $-1$ and that $h_j(s, L_{\text{min}, j})\ge 0$ for every $s\ge 0$ (again by (H$_3$)). This yields $(h_j(s, L_{\text{min}, j}),-1) = -h_j(s, L_{\text{min}, j}) \le 0$ as needed.  }
In order to prove the upper bound in \eqref{dis-L}, we fix $t \in (0,T)$, integrate \eqref{free-bdry2}, and use (H$_3$)-(i) to have
\begin{equation*}
    L_j(t)= L_j^{(0)}+ \int_0^t h_j(\Lambda_j(s), L_j(s)) \,\ds \le L_j^{(0)}+ T\, \|h_j\|_{L^{\infty}(\R^2)}.
    \end{equation*}
In addition, we observe that \eqref{free-bdry2} gives $|L_j'(t)| \le \|h_j\|_{L^{\infty}(\R^2)}$. Then, the fact that $L_j(t) \ge L_{\text{min}, j}$ allows us to conclude \eqref{rem:L}.
\end{proof}

\begin{lemma}
\label{lemma:B2}

The operator $\mathcal{B}_2 \colon C([0, T])^3 \to C([0, T])^2$ that maps $\boldsymbol{\Lambda} $ to the solution $\boldsymbol{L}$ to \eqref{free-bdry2} is Lipschitz continuous in the sense of
 \begin{equation}
        \label{lip-2}
        \|\boldsymbol{L}^{(1)}- \boldsymbol{L}^{(2)}\|_{C([0, T])^2} \le T\,
        \max_{j= 1,2}\bigl\{L_{h_j}\, e^{2\,T\,L_{h_j}} \bigr\}
         \|\boldsymbol{\Lambda}^{(1)}- \boldsymbol{\Lambda}^{(2)}\|_{C([0, T])^3},
    \end{equation}
where $L_{h_j}$ is the Lipschitz constant of $h_j$. If $T \,\max_{j= 1,2}\bigl\{L_{h_j}\, e^{2\,T\,L_{h_j}} \bigr\}< 1$, then $\mathcal{B}_2$ is contractive.
\end{lemma}

\begin{proof}

The proof works as for Lemma \ref{lemma:B1} so we omit it.
%
\end{proof}

We next investigate the existence of solutions to system \eqref{weak2.1}--\eqref{weak2.2}.

\begin{theorem} 
\label{lemma:syst}

Let $\boldsymbol{\Lambda}$ and $ \boldsymbol{L} $ be the unique solution to \eqref{linear-ode} and \eqref{free-bdry2}, respectively. 
Then, there exists a unique solution $(\bbf_1, \bbf_2) \in X$ to \eqref{weak2.1}--\eqref{weak2.2} such that $f_{i,j}(t,y)\in [0,1]$ for a.e. $y \in (0, 1)$ and $t \in [0, T]$.
\end{theorem}

\begin{proof}

To simplify the notation, in this proof we will drop the use of the $j$-index
and, for the reader's convenience, we split  the proof in several steps.

\medskip

%
\underline{Step 1: Approximation by truncation}.
Given a generic function $a$ we   introduce the truncation 
    \begin{equation}
       \label{truncation}
        a_{\text{tr}}=
        \begin{cases}
            a & \text{if } 0 \le a \le 1, \\
           0 & \text{otherwise}.
        \end{cases}
    \end{equation}
{We apply this to the non-linear transport terms $f_{\pm}\,(1-\rho)$ in \eqref{weak2} which yields (after summing up)}
    \begin{equation}
       \label{trunc}
        \begin{split}
            &\sum_{i= +,-}\int_0^1 \partial_t  f_i \,\varphi_i \dy+ \frac{D_T}{L^2(t)} \,\sum_{i= +,-}  \int_0^ 1 \partial_y  f_i \, \partial_y \varphi_i \dy \\
            &= \int_0^1 \frac{v_0}{L(t)} (f_+ \,(1-\rho))_{\text{tr}} \, \partial_y \varphi_+ - \frac{v_0}{L(t)}\, (f_-\,(1-\rho))_{\text{tr}} \, \partial_y \varphi_- \dy \\
            & \quad - \frac{L'(t)}{L(t)}\,   \sum_{i= +,-} \int_0^1    \left(y \,f_i \,\partial_y \varphi_i +  f_i\, \varphi_i \right)\dy+ \int_0^1 \lambda \,[(f_- - f_+)\, \varphi_+ + (f_+- f_-)\,\varphi_-]  \dy \\
            & \quad +\frac{1}{L(t)} \Big(-  \beta_{+}(\Lambda(t)) \, f_{+}(t, 1)\,\varphi_+(1) + \alpha_{+}(\Lambda_{\text{som}}(t))\, g_{+}(\bbf(t, 0))\,\varphi_+(0)  \\
            & \quad\qquad +\alpha_{-}(\Lambda(t))\, g_{-}(\bbf(t, 1))\,\varphi_-(1) - \beta_{-}(\Lambda_{\text{som}}(t)) \, f_{-}(t, 0)\,\varphi_-(0)\Big).
        \end{split}
    \end{equation}
We solve \eqref{trunc} by means  of  the Banach fixed point theorem. We follow  \cite{Jan-chemotaxis},  pointing out that a similar approach has been used also in \cite{Greta-Jan-Alois}, yet in a different context.

Let us set $Y:= (L^{\infty}((0,T); L^2(0,1)))^2$ and introduce the following nonempty, closed set
    \[
    \mathcal M= \l\{{\bbf}= (f_+, f_-) \in Y: \, \|{\bbf}\|_{Y} \le C_{\mathcal M}\r\},
    \]
with $T, C_{\mathcal M}>0$ to be specified. Then we define the mapping $\Phi\colon \mathcal M \to Y$ such that $\Phi(\widetilde {\bbf})= {\bbf}$
where, for fixed $\widetilde \bbf \in \mathcal{M}$, $\bbf$ solves the following linearized equation (cf. \cite[Chapter III]{Lieberman})
    \begin{equation}
        \label{weak5-linear}
        \begin{split}
            &\sum_{i= +,-}\int_0^1 \partial_t  f_i \,\varphi_i \dy+ \frac{D_T}{L^2(t)} \,\sum_{i= +,-}  \int_0^ 1 \partial_y  f_i \, \partial_y \varphi_i \dy \\
            &= \int_0^1 \frac{v_0}{L(t)} \,(\widetilde f_+\, (1-\widetilde \rho))_{\text{tr}} \, \partial_y \varphi_+ - \frac{v_0}{L(t)} \,(\widetilde f_-\,(1-\widetilde \rho))_{\text{tr}} \, \partial_y \varphi_- \dy \\
            & \quad - \frac{L'(t)}{L(t)}  \, \sum_{i= +,-} \int_0^1    y\, f_i \,\partial_y \varphi_i +  f_i \,\varphi_i \dy+ \int_0^1 \lambda\, [(f_- - f_+)\, \varphi_+ +(f_+- f_-)\,\varphi_-]  \dy \\
            & \quad +\frac{1}{L(t)}\, \Big(-\beta_{+}(\Lambda(t)) \, f_{+}(t, 1)\,\varphi_+(1) + \alpha_{+}(\Lambda_{\text{som}}(t))\, g_{+}(\bbf(t, 0))\,\varphi_+(0)  \\
            & \quad\qquad +\alpha_{-}(\Lambda(t))\, g_{-}(\bbf(t, 1))\,\varphi_-(1) - \beta_{-}(\Lambda_{\text{som}}(t)) \, f_{-}(t, 0)\,\varphi_-(0)\Big).
        \end{split}
    \end{equation}

\medskip

\noindent \underline{Step 2:  $\Phi$ is self-mapping}.  
We show that
   \begin{equation}
      \label{sm-1}
        \|{\bbf}\|_Y \le C_{\mathcal M}.
    \end{equation}
We choose $\varphi_i= f_i$ in \eqref{weak5-linear} and estimate the several terms appearing in the resulting equation separately. 
From \eqref{dis-L}, on the left-hand side we have
    \begin{equation*}
    \begin{split}
        & \frac{1}{2} \frac{\textup{d}}{\dt} \sum_{i= +,-} \int_0^1 |f_i|^2  \dy+ \frac{D_T}{L^2(t)} \sum_{i= +,-}  \int_0^ 1 |\partial_y  f_i|^2 \dy \\
        &  \ge \frac{1}{2} \frac{\textup{d}}{\dt} \sum_{i= +,-} \int_0^1 |f_i|^2  \dy+ \frac{D_T}{(L^{(0)}+ T\, \|h\|_{L^\infty(\R)})^2} \sum_{i= +,-}  \int_0^ 1 |\partial_y  f_i|^2 \dy.
        \end{split}
    \end{equation*}
On the right-hand side we first use equation \eqref{rem:L} along with  Young's inequality for some $\eps_1>0$ and the fact that $y \in (0,1)$ to achieve
    \begin{equation*}
        \sum_{i= +,-} \int_0^1 \frac{L'(t)}{L(t)} \,y\,f_i\,\partial_y f_i \dy \le \sum_{i= +,-} \l(\eps_1 \,\|\partial_y f_i\|_{L^2(0,1)}^2+ \frac{\|h\|_{L^\infty(\R)}^2}{2\,\eps_1 \,L_{\text{min}}^2} \|f_i\|_{L^2(0,1)}^2 \r),
    \end{equation*}
while \eqref{rem:L} once again gives
    \[
    -\frac{L'(t)}{L(t)} \sum_{i=+,-} \int_0^1 f_i^2 \dy \le \frac{\|h\|_{L^\infty(\R^2)}}{L_{\text{min}}} \sum_{i=+,-} \|f_i\|_{L^2(0,1)}^2.
    \]
 On the other hand, \eqref{dis-L},  Young's inequality for some $\eps_2>0$ and \eqref{truncation} give
    \begin{equation*}
     \pm\int_0^1 \frac{v_0}{L(t)}  \bigl(\widetilde f_i \,(1- \widetilde \rho)\bigr)_{\text{tr}}\, \partial_y f_i \dy \le C  + \eps_2 \|\partial_y f_i\|_{L^2(0,1)}^2.
    \end{equation*}
We further observe that
    \begin{equation*}
        \lambda \int_0^1 (f_- - f_+)\, f_+ + (f_+ - f_-)\,f_- \dy= -\lambda \int_0^1 (f_+ + f_-)^2 \dy \le 0.
    \end{equation*}
Finally we estimate the boundary terms. We use hypotheses (H$_2$), (H$_4$), (H$_5$), and equation \eqref{dis-L} together with  Young's inequality with some $\eps_3, \dots, \eps_6>0$ and the trace inequality \eqref{eq:trace_theorem} to achieve
    \begin{equation*}
    \begin{split}
     \frac{1}{L(t)}\, \beta_+(\Lambda(t)) \,f_+^2(t, 1)  & \le C \,\|f_+\|_{L^2(0,1)}^2+ \eps_3 \, \|\partial_y f_+\|_{L^2(0,1)}^2, \\
       \frac{1}{L(t)} \,\alpha_{+}(\Lambda_{\text{som}}(t))\,g_{+}(\widetilde \bbf_{+}(t, 0))\,f_+(t, 0)  & 
         \le C\,\|f_+\|_{L^2(0,1)}^2+ \eps_4 \, \|\partial f_+\|^2_{L^2(0,1)}, \\
     \frac{1}{L(t)} \,\alpha_{-}(\Lambda(t))\,g_{-}( \widetilde \bbf_{+}(t, 1))\,f_-(t,1) & \le C \, \|f_-\|_{L^2(0,1)}^2 + \eps_5\,\|\partial_y f_-\|^2_{L^2(0,1)}, \\
      \frac{1}{L(t)} \, \beta_{-}(\Lambda_{\text{som}}(t)) \,f_{-}^2(t, 0) & \le C \,
        \|f_-\|_{L^2(0,1)}^2
        + \eps_6\, \|\partial_y f_-\|_{L^2(0,1)}^2.
    \end{split}
    \end{equation*}
We choose $\eps_{\kappa}$, $\kappa=1, \dots, 6$, in such a way that all the terms of the form $\|\partial_y f_i\|_{L^2(0,1)}$ can be absorbed on the left-hand side of  \eqref{weak5-linear}, which  simplifies to 
    \begin{equation*}
         \frac{\textup{d}}{\dt} \sum_{i= +,-}  \int_0^1 |f_i|^2  \dy \le C \sum_{i= +,-} \|f_i\|_{L^2(0,1)}^2+ C.
    \end{equation*}
We then use a Gronwall argument  to infer
    \[
    \sup_{t \in (0, T)} \|f_i(t, \cdot)\|_{L^2(0,1)}^2 \le C=: C_{\mathcal M}^2.  
    \]
This implies that \eqref{sm-1} is satisfied and therefore $\Phi$ is self-mapping.

\medskip

\noindent \underline{Step 3: $\Phi$ is a contraction}.
Let $\widetilde {\bbf}_1, \widetilde {\bbf}_2 \in \mathcal M$ and let ${\bbf_1}=: \Phi(\widetilde {\bbf}_1)$ and ${\bbf_2}=:  \Phi(\widetilde {\bbf}_2)$ be two solutions to \eqref{weak5-linear} with the same initial datum $\bbf^0$. We then consider the difference of the  corresponding equations and  choose $\varphi_i= f_{i, 1}- f_{i, 2}$.
Reasoning as in Step 2
and exploiting the Lipschitz continuity of the functions 
    \begin{equation}
        \label{a+}
      \R^2 \ni (a,b)  \mapsto (a\,(1- a- b))_{\text{tr}} \quad \text{and} \quad  \R^2 \ni (a,b)  \mapsto (b\,(1- a- b))_{\text{tr}}
    \end{equation}
     we get
    \begin{equation*}
        \begin{split}
          &  \frac{\textup{d}}{\dt} \sum_{i= +,-}  \|f_{i, 1}- f_{i, 2}\|_{L^2(0,1)}^2 \le  C  \sum_{i= +,-} \Big(\|f_{i, 1}- f_{i, 2}\|_{L^2(0,1)}^2+\|\widetilde f_{i, 1}- \widetilde f_{i, 2}\|_{L^2(0,1)}^2\Big).
        \end{split}
    \end{equation*}
Again by means of a  Gronwall argument  we have
    \[
        \sum_{i= +,-} \|(f_{i, 1}- f_{i, 2})(t, \cdot)\|_{L^2(0,1)}^2 
         \le C\,T\, e^{CT} \sum_{i= +,-}  \|\widetilde f_{i, 1}- \widetilde f_{i, 2}\|_{L^{\infty}(0,T; L^2(0,1))}^2,
    \]
and then $\Phi$ is a contraction if  $T>0$ is small enough so that $ C\, T\, e^{C\, T}< 1$. Then Banach's fixed point theorem applies and we obtain a solution $\bbf \in (L^{\infty}(0,T; L^2(0,1)))^2 $ to \eqref{trunc}. 
A  standard regularity theory then gives  $\bbf \in (L^2(0,T; H^1(0,1)))^2$, with $\partial_t \bbf \in (L^2(0,T; H^1(0,1)')^2$.

\medskip

\underline{Step 4: Box constraints}. We  show that such $\bbf$ obtained in Step 4  is actually a solution to \eqref{weak2}, because it satisfies the box constraint $ f_+, f_- \ge 0$ and $\rho \le 1$.

We start by showing that $f_+ \ge 0$, and to this end we consider only the terms involving the $\varphi_+$-functions in \eqref{trunc}, that is, 
    \begin{equation}
        \label{eq-rho}
        \begin{split}
        &\int_0^1 \partial_t  f_+\,\varphi_+ \dy+ \frac{D_T}{L^2(t)}\, \int_0^1 \partial_y  f_+ \, \partial_y \varphi_+ \dy \\
            &=  \int_0^1 \Big( \frac{v_0}{L(t)} \, \bigl(f_+ \,(1- \rho)\bigr)_{\text{tr}} - \frac{L'(t)}{L(t)}\, y\, f_+ \Big) \partial_y \varphi_+ + \Big(\lambda\,( f_- -  f_+)- \frac{L'(t)}{L(t)} \,f_+ \Big)\,\varphi_+ \dy \\
            & \qquad +\frac{1}{L(t)}\, \Big(-  \beta_{+}(\Lambda(t))\,f_{+}(t, 1) \, \varphi_+(1)+ \alpha_{+}(\Lambda_{\text{som}}(t))\,g_{+}( \bbf(t, 0))\,\varphi_+(0)\Big).
        \end{split}
    \end{equation}
For every $\eps>0$ we consider the  function $\eta_\eps \in W^{2, \infty}(\R)$ given by
    \begin{equation}
     \label{der-eta}
    \eta_\eps(u)=
    \begin{cases}
        0 & \text{if } u \le 0, \\
       \displaystyle \frac{u^2}{4\eps} & \text{if } 0< u \le 2\eps, \\
      u- \eps & \text{if } u > 2\eps,
    \end{cases}
    \qquad \text{with} \quad
   \eta''_\eps(u)= 
   \begin{cases}
       0 & \text{if } u \le 0, \\
      \displaystyle \frac{1}{2\eps} & \text{if } 0< u \le 2\eps, \\
      0 & \text{if } u > 2\eps.
  \end{cases}
    \end{equation}
%
%
Next we choose $\varphi_+= -\eta_\eps'(-f_+)$   and observe that by the chain rule we have $\partial_y \varphi_+= \eta_\eps''(-f_+)\,\partial_y f_+$. Using such $\varphi_+$ in \eqref{eq-rho} gives
    \begin{equation*}
        \begin{split}
        & -\int_0^1 \partial_t f_+\, \eta_\eps'(-f_+) \dy+ \frac{D_T}{L(t)^2}  \int_0^ 1 \eta_\eps''(-f_+)\, |\partial_y  f_+|^2 \dy \\
            & =  \int_0^1 \Big(\frac{v_0}{L(t)} \, \bigl(f_+\, (1- \rho)\bigr)_{\text{tr}} - \frac{L'(t)}{L(t)}\, y\, f_+ \Big)  \,\eta''_\eps(-f_+) \,\partial_y f_+  
            - \Big(\lambda\,(f_- - f_+) - \frac{L'(t)}{L(t)} \,f_+ \Big)\, \eta'_\eps(-f_+) \dy \\
            & \qquad + \frac{1}{L(t)} \big( \beta_+(\Lambda(t))\, f_+(t,1)\, \eta_\eps'(-f_+(t, 1)) -\alpha_+(\Lambda_{\text{som}}(t)) \,g_+(\bbf(t,0))\,\eta_\eps'(-f_+(t,0))\big).
        \end{split}
    \end{equation*}
Thanks to   Young's inequality with a suitable $\kappa>0$ and to \eqref{truncation} we have
    \begin{equation*}
   \begin{split}
    &\int_0^1 \frac{v_0}{L(t)}\, \bigl(f_+\,(1-  \rho)\bigr)_{\text{tr}}\, \eta_\eps''(-f_+)\,\partial_y f_+ \dy \\
      &\qquad \le \kappa \int_0^1 \eta_\eps''(-f_+)\,|\partial_y f_+|^2 \dy+ \frac{1}{4\,\kappa} \int_0^1 \eta_\eps''(-f_+) \,\Big(\frac{v_0}{L(t)}\Big)^2 \,\bigl(f_+\, (1- \rho)\bigr)_{\text{tr}}^2 \dy,
    \end{split} 
    \end{equation*}
as well as
    \begin{align*}
       &  -\int_0^1 \frac{L'(t)}{L(t)}\, y\, f_+ \eta_\eps''(-f_+)\, \partial_y f_+ \dy \\
      & \qquad   \le \kappa \int_0^1 \eta_\eps''(-f_+)\, |\partial_y f_+|^2 \dy+ \frac{1}{\kappa} \,\Big(\frac{L'(t)}{L(t)}\Big)^2 \int_0^1 \eta_\eps''(-f_+)\, f_+^2 \dy,
    \end{align*}
where $\kappa$ depends on the lower and upper bounds on $L(t)$, see Lemma \ref{lemma:L}.
Choosing $\kappa$ sufficiently small and taking into account that the term involving $\alpha_+$ is non-negative we obtain
    \begin{equation}
       \label{eq-rho2}
      \begin{split}
           & \frac{\textup{d}}{\dt} \int_0^1 \eta_\eps(-f_+) \dy = -\int_0^1 \partial_t f_+ \,\eta_\eps'(-f_+) \dy \\
           & \le C \int_0^1  \, \eta_\eps''(-f_+) \Big[\Big(\frac{v_0}{L(t)}\Big)^2\, \bigl(f_+\, (1- \rho)\bigr)_{\text{tr}}^2+ \Big(\frac{L'(t)}{L(t)}\Big)^2\, f_+^2\Big]  -\l(\lambda\,(f_- - f_+)- \frac{L'(t)}{L(t)}\, f_+ \r) \eta'_\eps(-f_+) \dy \\
           & \qquad + \frac{1}{L(t)}\, \beta_+(\Lambda(t))\,f_+(t,1)\,\eta_\eps'(-f_+(t, 1)).
       \end{split}
   \end{equation}
 %
To gain some sign information on the right-hand side of \eqref{eq-rho2}, we introduce the set
    \begin{equation}
        \label{Om-eps}
    \Omega_\eps:= \{y \in (0,1): \, 0 < f_+(t, y), f_-(t, y) \le 2\eps\}.
    \end{equation}
We first use \eqref{dis-L}, \eqref{der-eta}, and the Lipschitz continuity of  \eqref{a+}
to have
    \begin{equation*}
\begin{split}
   &  \int_{\Omega_\eps} \eta_\eps''(-f_+)\, \Big(\frac{v_0}{L(t)}\Big)^2 \,\bigl(f_+\,(1-\rho)\bigr)_{\text{tr}}^2\; \dy \\
   &= \int_{\Omega_\eps} \eta_\eps''(-f_+)\, \Big(\frac{v_0}{L(t)}\Big)^2 \,\bigl(f_+\,(1-\rho)\bigr)_{\text{tr}} -(0 \cdot (1-\rho)\bigr)_{\text{tr}}\bigr) ^2  \dy 
     \le \frac{v_0^2}{L_{\text{min}}^2}\,2 \eps\, |\Omega_\eps|=: \tilde c_1\,\eps,
    \end{split}
    \end{equation*}
as well as
    \begin{align*}
        \int_{\Omega_\eps} \eta_\eps''(-f_+)\, \Big(\frac{L'(t)}{L(t)}\Big)^2\, f_+^2 \dy \le \frac{\|h\|_{L^\infty(\R^2)}^2}{L_{\text{min}}^2}\, |\Omega_\eps|\, 2 \eps:= \tilde c_2\, \eps,
    \end{align*}
and
    \begin{align*}
        \frac{L'(t)}{L(t)} \int_{\Omega_\eps} f_+ \,\eta_\eps'(\-f_+) \dy \le \frac{\|h\|_{L^\infty(\R^2)}^2}{L_{\text{min}}^2}\, |\Omega_\eps|\, 2 \eps= \tilde c_2\, \eps.
    \end{align*}
On the other hand, from \eqref{Om-eps} it follows that $-2\eps< f_+- f_- \le 2\eps$, which implies 
    \begin{equation*}
    -\int_{\Omega_\eps} \lambda\,(f_- - f_+) \dy \le 2\eps\,\lambda\,|\Omega_\eps|=: \tilde c_3\,\eps.
    \end{equation*}
Concerning the boundary term, we note that by definition the function $\eta_\eps'$ is nonzero only if its argument is nonnegative, which in any case gives
    \[
    \frac{1}{L(t)}\, \beta_+(\Lambda)\,f_+(t,1)\, \eta_\eps'(-f_+(t, 1)) \le 0.
    \]
Summarizing, from \eqref{eq-rho2} we have
    \begin{align*}
        \frac{\textup{d}}{\textup{d} t} \int_0^1 (f_+)_- \dy= \lim_{\eps \to 0} \frac{\textup{d}}{\textup{d} t} \int_{\Omega_\eps} \eta_\eps(-f_+) \dy \le 0,
    \end{align*}
which implies
    \begin{align*}
        \int_0^1 (f_+)_- \dy \le \int_0^1 (f_+^0)_- \dy= 0,
    \end{align*}
where the last equality holds thanks to assumption (H$_1$). It follows that $f_+ \ge 0$.
%
%
The proof that $f_- \ge 0$ works in a similar way, starting again from equation \eqref{trunc} and taking into account only the terms involving the $\varphi_-$-functions. 

\medskip

We now show that  $\rho \le 1$ and to this aim
start from  \eqref{trunc} with $\varphi_i:= \varphi$, for $i=+,-$. This gives
    \begin{equation*}
        \begin{split}
             & \int_0^1 \partial_t  \rho\,\varphi \dy+ \frac{D_T}{L^2(t)}  \int_0^ 1 \partial_y  \rho\,\partial_y \varphi \dy \\
            &= -\frac{L'(t)}{L(t)} \int_0^1 y \,\rho\, \partial_y  \varphi + \rho\, \varphi \dy+ \frac{v_0}{L(t)} \int_0^1    \l(\bigl(f_+\,(1-  \rho)\bigr)_{\text{tr}} - \bigl(f_-\,(1-  \rho)\bigr)_{\text{tr}}\r) \partial_y \varphi \dy \\
            &\quad - \frac{1}{L(t)}\, \Big[ \big(\beta_{+}(\Lambda(t)) \,f_{+}(t, 1)-\alpha_{-}(\Lambda(t))\,g_{-}(\bbf(t, 1))\big) \,\varphi(1) \\
            & \quad \quad  +\big(\beta_{-}(\Lambda_{\textup{som}}(t))\,  f_{-}(t, 0)-  \alpha_{+}(\Lambda_{\text{som}}(t))\, g_{+}(\bbf(t, 0))\big) \, \varphi(0)\Big].
        \end{split}
    \end{equation*}
We  choose $\varphi= \eta_\eps'(\rho-1)$,   $\widetilde \Omega_\eps= \{y \in (0,1): \, 1 \le \rho(t, y)  \le 1+ 2\eps\}$, and reason as before exploiting hypothesis (H$_2$). Then  the limit as $\eps \to 0$ entails
    \[
    0 \le \int_0^1 (\rho-1)^+ \dy \le \int_0^1 (\rho^0-1)^+ \dy= 0,
    \]
thanks again to (H$_1$). It follows that  $\rho \le 1$, then the claim is proved. 

\medskip

\underline{Step 5: Conclusion}.
Using the original notation, we have shown the existence of a solution $\bar \bbf$ to \eqref{weak2} such that $\bar f_+, \bar f_- \ge 0$ and $\bar \rho \le 1$ for a.e. $y \in (0,1)$, $t \in [0, T]$. This implies that $\bbf$ is a weak solution to \eqref{weak2.1}--\eqref{weak2.2} such that $f_+, f_- \ge 0$ and $\rho \le 1$ for a.e. $x \in (0, L(t))$, $t \in [0, T]$. The system is completely solved. 
\end{proof}

\begin{lemma}
\label{lemma:B3}

The operator  $\mathcal{B}_3 \colon C([0, T])^3 \times C([0, T])^2 \to X$, which maps a pair 
$(\boldsymbol{\Lambda}, \boldsymbol{L})$ to the unique solution $(\bbf_1, \bbf_2)$ of \eqref{weak2.1}--\eqref{weak2.2},
is Lipschitz continuous. 
\end{lemma}

\begin{proof}

Thanks to Theorem \ref{lemma:syst}, $\mathcal{B}_3$ is well defined. We choose $\boldsymbol{\Lambda}^{(a)} \in C([0,T])^3$ and $\boldsymbol{L}^{(a)} \in C([0, T])^2 $ and set $(\bbf_1^{(a)}, \bbf_2^{(a)}
):= \mathcal{B}_3(\boldsymbol{\Lambda}^{(a)}, \boldsymbol{L}^{(a)})$, $a=1,2$. Then we start from equations \eqref{weak2.1}, \eqref{weak2.2} for the respective copies, take the corresponding differences and use a Gronwall inequality to achieve
    \begin{align}
    \label{lip-3.1}      
        &\|(\bbf_1^{(1)}, \bbf_2^{(1)})- (\bbf_1^{(2)}, \bbf_2^{(2)})\|_{\prod_{j=1,2} (L^2(0,T; H^1(0, 1)))^2} \nonumber\\
        &\qquad \le \overline{C}\,
        e^{CT}\,\Big(\|\boldsymbol{\Lambda}^{(1)}-\boldsymbol{\Lambda}^{(2)}\|_{C([0,T])^3}^2 
        +  \|\boldsymbol{L}^{(1)}- \boldsymbol{L}^{(2)}\|_{C([0, T])^2}^2\Big),
    \end{align}
where the constant $\overline C $  is shown to be uniform. 
%
\end{proof}

\begin{proof}[{Proof of Theorem \ref{thm:existence}}]

We use the Banach fixed point theorem to show the existence of a unique solution 
$(\bbf_1, \bbf_2, \boldsymbol{\Lambda}, \boldsymbol{L})$ to \eqref{eq:model}--\eqref{ode}.
We consider the set
    \[
    \mathcal{K}= \{(\bbf_1, \bbf_2
    ) \in X : \,
    0 \le f_{i, j}(t,x) \le 1  \text{ for a.e. $x \in (0, 1)$ and $t \in [0, T]$}\}.     
    \]
Let $\mathcal{B} \colon \mathcal{K} \to X $ be  given by $\mathcal{B}(\bbf_1, \bbf_2
) = \mathcal{B}_3(\mathcal{B}_1(\bbf_1, \bbf_2
), \mathcal{B}_2(\mathcal{B}_1(\bbf_1, \bbf_2
)))$, 
where $\mathcal{B}_1, \mathcal{B}_2, \mathcal{B}_3$ are defined in Lemmas~\ref{lemma:B1}, \ref{lemma:B2}, and \ref{lemma:B3}, respectively. Thus $\mathcal{B}$ is well defined, while  Theorem \ref{lemma:syst} implies that $\mathcal{B}$ is  self-mapping, that is, $\mathcal{B}\colon \mathcal{K} \to \mathcal{K}$.

We next show that $\mathcal{B}$ is a contraction. Let $(\widehat \bbf_1^{(1)}, \widehat \bbf_2^{(1)}
), (\widehat \bbf_1^{(2)}, \widehat \bbf_2^{(2)}
) \in \mathcal{K}$ and set $( \bbf_1^{(1)}, \bbf_2^{(1)}
):= \mathcal{B}(\widehat \bbf_1^{(1)}, \widehat \bbf_2^{(1)}
)$ as well as $(\bbf_1^{(2)},  \bbf_2^{(2)}
):= \mathcal{B}(\widehat \bbf_1^{(2)}, \widehat \bbf_2^{(2)}
)$. Using \eqref{lip-3.1}, \eqref{lip-2}, and \eqref{lip-1}, and summarizing the not essential constants gives 
    \begin{equation*}
\begin{split}
       & \|(\bbf_1^{(1)}, \bbf_2^{(1)}
       )- (\bbf_1^{(2)}, \bbf_2^{(2)}
       )\|_X  \\
        &  \le C T  e^{CT} \big(T \max_{j=1,2} \big\{L_{h_j} e^{2T L_{h_j}}\big\}+1 \big)    \|(\widehat \bbf_1^{(1)}, \widehat \bbf_2^{(1)}
        )- (\widehat \bbf_1^{(2)}, \widehat \bbf_2^{(2)}
        )\|_X,
        \end{split}
    \end{equation*}
from which the contractivity follows if $T>0$ is small enough that $C T e^{CT}< 1$.
Thanks to the Banach's fixed point theorem  we then infer the existence of a unique $( \bbf_1, \bbf_2
) \in \mathcal{K}$ such that $(\bbf_1, \bbf_2
)= \mathcal{B}( \bbf_1, \bbf_2
)$.
From \eqref{lip-1}, this implies the uniqueness of $\boldsymbol{\Lambda}$, too. 
Due to the uniform $L^\infty$-bounds on $\bbf_j$, a concatenation argument yields existence on the complete interval $[0,T]$. The proof is thus complete. 
\end{proof}

\section{Stationary solutions}
\label{sect4}
This section is dedicated to a brief investigation of stationary solutions $(\bbf_1^{\infty}, \bbf_2^{\infty}, \boldsymbol{\Lambda}^{\infty}, \boldsymbol{L}^{\infty})$ to \eqref{eq:model}--\eqref{ode}. That is, they satisfy 
%
%
    \begin{equation} \label{eq:model-stat}
    \begin{aligned}
    0&= -\frac{1}{(L_j^\infty)^2}\,\partial_x  \left[\phantom{-}L_j^\infty\,v_0\,f_{+,j}^{\infty}\,(1- \rho_j^{\infty})- D_T\,\partial_{x} f_{+,j}^{\infty} \right] + \lambda\,(f_{-,j}^{\infty} - f_{+,j}^{\infty}), \\
    0&= -\frac{1}{(L_j^\infty)^2}\,\partial_x \left[-L_j^\infty\, v_0\,f_{-,j}^{\infty}\,(1- \rho_j^{\infty})- D_T\,\partial_{x} f_{-,j}^{\infty} \right] + \lambda\, (f_{+,j}^{\infty} - f_{-,j}^{\infty}),
    \end{aligned}
\quad \text{in }  (0, 1),
    \end{equation}
with boundary conditions
    \begin{equation}
    \label{boundary-stat}
    \begin{aligned}
     J_{+,j}^{\infty}(0)  &= L_j^\infty\alpha_{+,j}(\Lambda_{\text{som}}^\infty)\, g_{+,j}(\bbf_{j}^{\infty}(0)), &\quad
    -J_{-,j}^{\infty}(0)&= L_j^\infty\beta_{-,j}(\Lambda_{\text{som}}^{\infty})\,f_{-,j}^{\infty}(0),\\
    J_{+,j}^{\infty}(1)  &= L_j^\infty\beta_{+,j}(\Lambda_j^\infty)\,f_{+,j}^{\infty}(1),  
    &\quad -J_{-,j}^{\infty}(1) &= L_j^\infty\alpha_{-,j}(\Lambda_j^\infty)\,g_{-,j}(\bbf_{j}^{\infty}(1)),
    \end{aligned}
    \end{equation}
while $\boldsymbol{\Lambda}^\infty$ and $\boldsymbol{L}^\infty$ solve
    \begin{equation}
    \label{ode-stat}
    \begin{aligned}
    0 &=  \sum_{j=1,2}  \bigl( \beta_{-,j}(\Lambda_{\text{som}}^{\infty})\,f_{-,j}^{\infty}(0) - \alpha_{+,j}(\Lambda_{\text{som}}^\infty)\, g_{+,j}(\bbf_{j}^{\infty}(0)) \bigr), \\ 
    0 &=  \beta_{+,j}(\Lambda_j^\infty)\,f_{+,j}^{\infty}(L_j^\infty)- \alpha_{-,j}(\Lambda_j^\infty)\, g_{-,j}(\bbf_{j}^{\infty}( L_j^\infty)) - \chi\,h_j(\Lambda_j^\infty, L_j^\infty), \\
            0 &= h_j(\Lambda_j^\infty, L_j^\infty), 
    \end{aligned}
    \end{equation}
respectively. Notice that we assumed $\gamma_{\text{prod}}(t) \to 0$ as $t\to\infty$, cf. hypothesis (H$_7$), in order to guarantee a finite total mass of the stationary state {(which, in turn, results in finite values of $L_j^\infty$)}. From the modelling point of view, this means that at the end of the growth phase, when the neuron is fully developed, there is no more production of vesicles in the soma. In addition to equations \eqref{eq:model-stat}--\eqref{ode-stat}, stationary solutions are parametrized by their total mass
$$
m_\infty := \sum_{j=1,2} \left(L_j^\infty\int_0^{1} \rho_j^\infty(y)\dy + \Lambda_j^\infty {+ \chi\,L_j^\infty}\right) + \Lambda_{\text{som}}^\infty.
$$
For fixed $0< m_\infty < +\infty$, we expect three possible types of stationary solutions {(where the upper bound on $m_\infty$ excludes neurites of infinite length)}:
\,

\noindent $\bullet$ No mass inside the neurites, i.e., $\rho_j^\infty = 0$, and $m_\infty = \Lambda_1^\infty+ \Lambda_2^\infty + \chi\,L_1^\infty + \chi\,L_2^\infty + \Lambda_{\text{som}}^\infty$. This solution is always possible as it automatically satisfies \eqref{boundary-stat}. The length depends on the fraction of mass stored in each $\Lambda_j$, which yields a family of infinite solutions.

\noindent $\bullet$ Constant  solutions with mass inside the neurites, i.e., $\bbf_j^\infty \neq (0,0)$. In this case, for $\lambda > 0$, the reaction term enforces $f_{-,j}^\infty = f_{+,j}^\infty=: f_j^\infty$. However, such solutions only exist if the nonlinearities at the boundary satisfy   conditions so that \eqref{boundary-stat} holds. In this case, compatibility with a given total mass $m_\infty$ can be obtained by adjusting the concentration  $\Lambda_{\text{som}}^\infty$, which decouples from the remaining equations. 

\noindent $\bullet$ Non-constant  solutions,  featuring boundary layers at the end of the neurites.

A natural question is the existence of 
non-constant
stationary solutions as well as their uniqueness and stability properties which we postpone to future work. Instead, we focus on conditions for the existence of non-trivial constant solutions. {Let us note, however, that even if the biological system reaches a stationary state in terms of the length of the neurites, we would still expect a flux of vesicles through the system (i.e., the system will still be out of thermodynamic equilibrium). Thus, the fluxes at the boundary would be non-zero and one would expect non-constant stationary solutions even in this case.}

\subsection{Constant stationary solutions}
We assume a strictly positive reaction rate $\lambda > 0$ which requires $f_{+,j} = f_{-,j} =: f_j^\infty \in [0,1]$. {Assuming in addition $f_j^\infty = \textrm{const}.$, \eqref{eq:model-stat} is automatically satisfied (note that $(L_j^\infty)'=0$). Thus, the actual constants are determined via the total mass, the stationary solutions to the ODEs and the boundary couplings, only, where 
the fluxes take the form}
    \begin{equation}
        \label{J-infty}
    J_{+,j}^{\infty}= -J_{-,j}^{\infty}= v_0\,f_{j}^{\infty}\,(1- \rho_j^{\infty})= v_0\,f_{j}^{\infty}\,(1- 2f_j^{\infty}).
    \end{equation}
    Making the choice
\begin{align}\label{eq:choice_g}
    g_{+,j}(f_{+,j}, f_{-,j})= f_{+,j}\,(1- \rho_j) \quad \text{as well as} \quad g_{-,j}(f_{+,j}, f_{-,j})= f_{-,j}\,(1- \rho_j),
    \end{align}
 the boundary conditions \eqref{boundary-stat} become
    \begin{equation}
        \label{boundary-stat-inf}
    \begin{split}
    J_{+,j}^{\infty}&= {L_j^\infty}\alpha_{+,j}(\Lambda_{\text{som}}^\infty)\, f_j^{\infty}\,(1- \rho_j^{\infty})= {L_j^\infty}\beta_{+,j}(\Lambda_j^{\infty})\, f_j^{\infty}, \\ -J_{-,j}^{\infty}&=   {L_j^\infty}\alpha_{-,j}(\Lambda_{j}^\infty)\,f_j^{\infty}\,(1- \rho_j^{\infty})= {L_j^\infty}\beta_{-,j}(\Lambda_{\text{som}}^{\infty})\, f_j^{\infty},
    \end{split}
    \end{equation}
which together with \eqref{J-infty} yield
    \begin{equation*}
    \alpha_{+,j}(\Lambda_{\text{som}}^{\infty})= \alpha_{-,j}(\Lambda_j^{\infty})= \frac{v_0}{L_j^\infty} \quad \text{as well as} \quad \beta_{-,j}(\Lambda_{\text{som}}^{\infty})= \beta_{+,j}(\Lambda_j^{\infty})= \frac{v_0(1- \rho_j^{\infty})}{L_j^\infty}.
     \end{equation*}
%
 %
%
%
%
%
%
%
Thus, fixing the values of $f_j^\infty$, we obtain $\Lambda_j^\infty$, $\Lambda_{\text{som}}^\infty$ by inverting $\alpha_{+,j}$, {$\alpha_{-,j}$}, $\beta_{+,j}$, and $\beta_{-,j}$, respectively, together with the compatibility conditions 
$$
\alpha_{-,j}(\Lambda_j^\infty)(1-\rho_j^\infty) = \beta_{+,j}(\Lambda_j^\infty) \quad \text{as well as} \quad \alpha_{+,j}(\Lambda_{\text{som}}^\infty)(1- \rho_j^\infty)= \beta_{-,j}(\Lambda_{\text{som}}^\infty).
$$
We further observe that  the equations in \eqref{ode-stat} are the differences of in- and outflow at the respective boundaries, which in the case of constant $f_j^\infty$-s read as
    \begin{equation}
        \label{ode-const}
        \begin{split}
            0&= \sum_{j=1,2} \l(\beta_{-,j}(\Lambda_{\text{som}}^\infty)\,f_j^\infty -  \alpha_{+,j}(\Lambda_{\text{som}}^\infty)\,f_j^\infty\,(1- \rho_j^\infty)\r), \\
            0&= \beta_{+,j}(\Lambda_j^\infty)\,f_j^\infty - \alpha_{-,j}(\Lambda_j^\infty)\,f_j^\infty\,(1- \rho_j^\infty) - \chi\,h_j(\Lambda_j^\infty, L_j^\infty).
        \end{split}
    \end{equation}
It turns out that \eqref{ode-const} will be automatically satisfied as soon as \eqref{boundary-stat-inf} and  \eqref{ode-stat} hold.
In particular, the last equation in \eqref{ode-stat} will determine the values of $L_j^\infty$. We make the choices
\begin{align}
\label{eq:choice_beta}
            \beta_{+,j}(s)&= c_{\beta_{+,j}} \l(1- \frac{s}{\Lambda_{j, \text{max}}}\r), \quad \beta_{-,j}(s)= c_{\beta_{-,j}} \l(1- \frac{s}{\Lambda_{\text{som}, \text{max}}}\r), \\
            \alpha_{+,j}(s)&= c_{\alpha_{+,j}}  \frac{s}{\Lambda_{\text{som}, \text{max}}}, \quad \alpha_{-,j}(s)= c_{\alpha_{-,j}}  \frac{s}{\Lambda_{j, \text{max}}}, \label{eq:choice_alpha}
            \end{align}
        for some $c_{\beta_{i,j}}, c_{\alpha_{i,j}} \ge 0$, with $\Lambda_{\text{som}, \text{max}}$, $\Lambda_{j,\text{max}}$ being the maximal capacity of soma and growth cones.
        Then, for given $f_j^\infty \in (0,\frac{1}{2}]$ to be a stationary solution, we need
        \begin{align*}
        c_{\alpha_{+,j}}&= {\frac{v_0}{L_j^\infty}} \frac{\Lambda_{\text{som,max}}}{\Lambda_{\text{som}}^\infty}, \quad c_{\beta_{-,j}}= {\frac{v_0}{L_j^\infty}}(1- \rho_j^\infty) \frac{\Lambda_{\text{som,max}}}{\Lambda_{\text{som,max}}- \Lambda_{\text{som}}^\infty}, \\ c_{\alpha_{-,j}}&= {\frac{v_0}{L_j^\infty}} \frac{\Lambda_{j, \text{max}}}{\Lambda_j^\infty}, \quad c_{\beta_{+,j}}= {\frac{v_0}{L_j^\infty}}(1- \rho_j^\infty) \frac{\Lambda_{j, \text{max}}}{\Lambda_{j, \text{max}}- \Lambda_j^\infty}.
        \end{align*}
The interesting question of stability of these states will be treated in future work.

\section{Finite volume scheme and scaling}
\label{sect5}


\subsection{Finite volume scheme}\label{sec:FV}
We now present a computational scheme for the numerical solution of model \eqref{eq:model}--\eqref{ode}. The scheme relies on a spatial finite volume discretization of the conservation law \eqref{eq:model_cont} and adapted implicit-explicit time stepping schemes.
Starting point for the construction of a discretization is the transformed equation \eqref{eq:discr1}. First, we introduce an equidistant grid 
\begin{equation*}
    0 = y_{-1/2} < y_{1/2} < \ldots < y_{n_e-1/2} < y_{n_e+1/2} = 1
\end{equation*}
of the interval $(0,1)$ and define control volumes $I_k := (y_{k-1/2},y_{k+1/2})$, $k=0,\ldots,n_e$. The mesh parameter is $h=y_{k+1/2}-y_{k-1/2} = (n_e+1)^{-1}$. The cell averages of the approximate (transformed) solution are denoted by
\begin{equation*}
    \bar f_{i,j}^k(t) := \frac1h \int_{I_k}  f_{i,j}(t,y)\dy,\qquad k=0,\ldots,n_e,
\end{equation*}
$i\in\{+,-\}$, $j\in \{1,2\}$.
To shorten the notation we omit the vesicle index $j\in\{1,2\}$ in the following.
Integrating  \eqref{eq:discr1} over an arbitrary control volume $I_k$, $k=0, \ldots, n_e$, yields
\begin{align*}
  0 = \int_{I_k} \Big[
    \partial_t f_+ + \frac{1}{L(t)^2} \partial_y\left[- D_T\,\partial_y f_+ + L(t)\,((v_0\,(1-\rho) - L'(t)\,L(t)\,y)\,f_+)\right]&\nonumber\\
    + \frac{L'(t)}{L(t)}\, f_+ - \lambda\,(f_--f_+)
    \Big]\dy &\nonumber\\
    = h\,\partial_t \bar f_+^k + \left[- \frac{D_T}{L(t)^2}\,\partial_y f_+ + \frac1{L(t)}\, \left((v_0\,(1-\rho) - L'(t)\,y)\,f_+\right)\right]_{y_{k-1/2}}^{y_{k+1/2}} &\nonumber\\
    + h\,\frac{L'(t)}{L(t)}\, \bar f_+^k - h\,\lambda\,(\bar f_-^k-\bar f_+^k)&
    . 
\end{align*}
We denote the convective flux by $\mathbf{v}_+(t,y) := v_0\,(1-\rho)-L'(t)\,y$ and use a Lax-Friedrichs approximation at the endpoints of the control volume
\begin{equation*}
     F_+^{k+1/2}(t) := \avg{(\mathbf{v}_+\,\bar f_+)(t,y_{k+1/2})} - \frac{1}{2}\,\jump{\bar f_+(t,y_{k+1/2})} \approx (\mathbf{v}_+\,f_+)(t,y_{k+1/2}),
\end{equation*}
with $\avg{\cdot}$ and $\jump{\cdot}$ denoting the usual average and jump operators.
For the diffusive fluxes we use an approximation by central differences
\begin{equation*}
    \partial_y f_+(t,y_{k+1/2}) \approx \frac{\bar f_+^{k+1}(t) - \bar f_+^{k}(t)}{h}.
\end{equation*}
For the inner intervals $I_k$ with $k \in \{1,\ldots,n_e-1\}$ this gives the equations
\begin{subequations}\label{eq:fvm_f_plus}
\begin{align}
   \partial_t \bar f_+^k + \frac{D_T}{(h\,L)^2}\left(-\bar f_+^{k-1} + 2\bar f_+^k - \bar f_+^{k+1}\right) 
    &= \frac1{h\,L} \left(F_+^{k-1/2}-F_+^{k+1/2}\right) \nonumber\\
    & \quad + \lambda\,(\bar f_-^k - \bar f_+^k) - \frac{L'}{L}\,\bar f_+^k,
\end{align}
while for $k=0$ and $k=n_e$ we insert the boundary conditions \eqref{boundary_ref_domain} to obtain
\begin{align}
  \partial_t\bar f_+^0 + \frac{D_T}{(h\,L)^2}\,(\bar f_+^0 - \bar f_+^1) 
  &= \frac1{h\,L}\left(\alpha_+(\Lambda_{\text{som}})\,g_+(\bar f_+^0, \bar f_-^0) - F_+^{1/2} \right) \nonumber\\
 & \quad + \lambda\,(\bar f_-^0 - \bar f_+^0) - \frac{L'}{L}\,\bar f_+^0, \\
  \partial_t\bar f_+^{n_e} + \frac{D_T}{(h\,L)^2}\,(\bar f_+^{n_e} - \bar f_+^{n_e-1})  
    &= \frac1{h\,L} \left(F_+^{n_e-1/2} - \beta_+(\Lambda)\,\bar f_+^{n_e}
    \right) \nonumber\\
    & \quad + \lambda\,(\bar f_-^{n_e} - \bar f_+^{n_e}) - \frac{L'}{L}\,\bar f_+^{n_e},
\end{align}
\end{subequations}
almost everywhere in $(0,T)$.
In the same way we deduce a semi-discrete system for $\bar f_-^k$, $k=0,\ldots,n_e$, taking into account the corresponding boundary conditions from \eqref{boundary_ref_domain}.

To treat the time-dependency we use an implicit-explicit time-stepping scheme. We introduce a time grid $t_n = \tau\,n$, for $n=0,\ldots,n_t$, and for some time-dependent function $g\colon [0,T]\to X$ we use the notation
    $g(t_n) =: g^{(n)}$.
To deduce a fully-discrete scheme we replace the time derivatives in \eqref{eq:fvm_f_plus} by a difference quotient $\partial_t\bar f_{+}(t_{n+1}) \approx \tau^{-1}(\bar f_+^{(n+1)}-\bar f_+^{(n)})$ and evaluate the remaining terms related to diffusion in the successive time-point $t_{n+1}$ and all convection and reaction related terms in the current time point $t_n$. This yields a system of linear equations of the form
\begin{equation}\label{eq:full_discrete_f}
    \left(M+\tau\,\frac{D_T}{(L^{(n)})^2}\,A\right)\,\vec f_{\pm}^{(n+1)} = M\,\vec f_{\pm}^{(n)} + \tau\,\vec b_{\pm}^{(n)} + \tau\,\vec c_{\pm}^{(n)},
\end{equation}
with vector of unknowns $\vec f_\pm^{(n)} = (f_\pm^{0,(n)}, \ldots, f_\pm^{n_e,(n)})^\top$, mass matrix $M$, diffusion matrix $A$, a vector $\vec b_\pm^{(n)}$ summarizing the convection related terms and another vector for the reaction related terms $\vec c_\pm^{(n)}$. In the same way we deduce equations for the discretized 
ordinary differential equations \eqref{ode} and \eqref{free-boundary} which correspond to a standard backward Euler discretization:
\begin{align}
    \Lambda_{\text{som}}^{(n+1)} &=  \Lambda_{\text{som}}^{(n)} + \tau\,\sum_{j=1,2}  \left( \beta_{-,j}\,\bar f_{j,-}^{0, (n+1)} - \alpha_{+,j}(\Lambda_{\text{som}}^{(n+1)})\, g_{+}(\bar f_{+,j}^{0,(n+1)}, \bar f_{-,j}^{0,(n+1)}) \right) + \tau\,\gamma_{\text{prod}}^{(n)}, \label{eq:fvm_Lambda_som}\\
    \Lambda_j^{(n+1)} &= \Lambda_j^{(n)} + \tau\,\beta_{+,j}(\Lambda_j^{(n+1)})\,\bar f_{+,j}^{n_e,(n+1)} \label{eq:fvm_Lambda_j}\\
    &\quad -\tau\,\alpha_{-,j}(\Lambda_j^{(n+1)})\,g_{-,j}(\bar f_{+,j}^{n_e,(n+1)}, \bar f_{-,j}^{n_e,(n+1)}) - 
    \tau\,\chi\, h_j(\Lambda_j^{(n)}, L_j^{(n)}), \nonumber\\
L_j^{(n+1)} &= L_j^{(n)} + \tau\,h_j(\Lambda_j^{(n+1)}, L_j^{(n+1)}),\qquad j=1,2.\label{eq:fvm_L}
\end{align}
Equations \eqref{eq:full_discrete_f}--\eqref{eq:fvm_L} even decouple. One after the other, we can compute
\begin{equation*}
    \boldsymbol f^0, L_j^0, \Lambda_k^0\;\rightarrow\;
    \vec f_{\pm,j}^{(1)}\;\rightarrow\; \Lambda_{k}^{(1)}\;\rightarrow\; L_j^{(1)}
    \; \rightarrow\; \vec f_{\pm,j}^{(2)}\;\rightarrow\; \Lambda_k^{(2)}\;\rightarrow\; L_j^{(2)} \;\rightarrow \;\ldots
\end{equation*}
for $k\in \{1,2,\text{som}\}$ and $j\in\{1,2\}$.

\subsection{Non-dimensionalisation of the model}

To transform the model to a dimensionless form we introduce  a typical time scale $\tilde t$, a typical length $\tilde L$, etc., and dimensionless quantities $\bar t,\, \bar L$ such that $t = \tilde t \bar t$, $L = \tilde L \bar L$. {This is performed on the original system from Section \ref{sect2}, not on the one transformed to the unit interval, as we want to work with appropriate physical units for all quantities, including the length of the neurites.} Realistic typical values are taken from \cite{Humpert2021_transport} (see also \cite{Pfenninger2009, Urbina2018, Tsaneva-Atanasova2009}) which yield the following choices: the \textit{typical length} is $\tilde{L} = 25 ~\mu \text{m}$, the \textit{typical time} is $\tilde{t} = 7200 ~\text{s}$, the \textit{diffusion constant} is $D_T = 0.5~ \frac{\mu\text{m}^2}{\text{s}}$, the \textit{velocity} is $\tilde{v}_0 = 50 \frac{\mu\text{m}}{\text{min}} =  \frac56~\frac{\mu\text{m}}{\text{s}}$. For the \textit{reaction rate} we assume $\tilde \lambda = \frac{1}{s}$.
%
%
The \textit{typical influx and outflow velocity} is $\tilde{\alpha} = \tilde{\beta} = 0.4~ \frac{\mu\text{m}}{\text{s}}$. 
Finally we choose a typical production of $\tilde \gamma = 10\, \text{vesicles}/\text{sec}$.

The remaining quantities to be determined are the maximal density of vesicles inside the neurites, the factor which translates a given number of vesicles with length change of the neurite and the maximal capacity of soma and growth cones.
\subsubsection*{Maximal density}
We assume the neurite to be tube-shaped, pick a circular cross-section at an arbitrary point and calculate the maximal number of circles having the diameter of the vesicles that fit the circle whose diameter is that of the neurite. In this situation, hexagonal packing of the smaller circles is optimal, which allows to cover about $90\,\%$ of the area (neglecting boundary effects). As the \textit{typical diameter} of one vesicle is $130\,\rm{nm}$ and the neurite diameter is 1 $\mu$m, we obtain the condition
\begin{align*}
	\underbrace{0.9\cdot n_{\text{max}}\pi \left(\frac{130\,\rm{nm}}{2}\right)^2}_{\text{area covered by $n_{\text{max}}$ circles of vesicle diameter}}\le \underbrace{\pi \left(\frac{1000\,\rm{nm}}{2}\right)^2}_{\text{area of neurite cross-section}},
\end{align*}
which implies $n_{\text{max}} \le 65$. Now for a tube segment of length $1\,\mathrm{\mu m}$, one can stack  $7$ fully packed cross-section slices, each of which has the diameter of the vesicles, i.e., $130\,\rm{nm}$. This results in a maximal density of $455 \frac{\text{vesicles}}{\mu m}$. As the neurite also contains microtubules and as an optimal packing is biologically unrealistic, we take one third of this value as maximal density, which yields
$
\rho_{\text{max}} = 155 \frac{\text{vesicles}}{\mathrm{\mu m}}.
$
The \textit{typical density} of anterograde and retrograde particles is fixed to $\tilde f := \tilde{f}_+= \tilde{f}_- = 39~ \frac{\text{vesicles}}{\mu \text{m}}$, so that their sum corresponds to a half filled neurite. Thus, for the scaled variables $\bar f_+, \,\bar f_-$, their sum being $\bar \rho= \bar f_+ + \bar f_- = 2$ corresponds to a completely filled neuron. This implies that the term $1- \rho$ has to be replaced by $1-\frac{\bar \rho}{2}$. 
\subsubsection*{Vesicles and growth}
We again consider the neurite as a cylinder with a diameter of $1\,\mu\rm{m}$. Thus the surface area of a segment of length $1\,\mu\rm{m}$ is 
$
A_{\text{surf}} = 2\pi \frac{1\,\mu\rm{m}}{2}1\,\mu\rm{m} \approx 3.14\, (\mu\rm{m})^2.
$
We consider vesicles of $130$ nm diameter, which thus possess a surface area of $0.053$ $(\mu\rm{m})^2$. Thus the number of vesicles needed for an extension of $1\,\mu\rm{m}$ is 
$$
\frac{3.14\,  \mu\rm{m}^2}{0.053\,\mu\rm{m}^2} \approx 59,\quad\text{ i.e., we fix } c_1 = c_2 = c_h := \frac{58.4 \text{ vesicles}}{\mu\rm{m}}.
$$
\subsubsection*{Maximal capacities and minimal values}
Finally, we fix the maximal amount of vesicles in the pools and soma to $\Lambda_{\text{som,max}} = 6000\,\text{vesicles}$ and $\Lambda_{j,\text{max}}= 400\,\text{vesicles}$, and choose the typical values $\tilde \Lambda_{\text{som}}$ and $\tilde \Lambda_{\text{cone}}$ as half of the maximum, respectively. It remains to fix the minimal length of each neurite as well as the number of vesicles in the growth cone which defines the switching point between growth and shrinkage. We choose a minimal length of $5\,\mu\rm{m}$ while the sign of $h_j$ changes when the number of vesicles in the growth cone reaches a value of $100\,\text{vesicles}$. This yields the dimensionless quantities 
$
\bar \Lambda_{\text{growth}} = 1, \; \bar L_{\text{min}} = 0.1.
$
Applying the scaling,  model \eqref{eq:model}--\eqref{ode} transforms to
    \begin{align*} 
    \begin{aligned}
    \partial_t \bar f_{+,j}& +\partial_x  \left(\kappa_v\,\bar f_{+,j}\left(1-  \frac{\bar \rho_j}{2}\right)- \kappa_D\,\partial_x \bar f_{+,j}\right) = \kappa_\lambda\bar\lambda (\bar f_{-,j} - \bar f_{+,j}), \\
    \partial_t \bar f_{-,j}& - \partial_x\left(\kappa_v\,\bar f_{-,j}\left(1-  \frac{\bar \rho_j}{2}\right)- \kappa_D\,\partial_x \bar f_{-,j}\right) =  \kappa_\lambda \bar \lambda(\bar f_{+,j} - \bar f_{-,j}),
    \end{aligned}
    \end{align*}
$\text{in } (0, T) \times (0, \bar L_j(t)) $, with dimensionless parameters
    $
    \kappa_v = \frac{v_0 \tilde t }{\tilde L}, \; \kappa_D = D_T\frac{\tilde t }{\tilde L^2}, \; \kappa_\lambda = \tilde t \tilde \lambda, 
    $
and with boundary conditions (keeping the choices \eqref{eq:choice_g}, \eqref{eq:choice_beta}, \eqref{eq:choice_alpha})
    \begin{equation*}
    \begin{aligned}
     \bar J_{+,j}(t,0) &= \kappa_{\alpha_{+,j}}\frac{\bar \Lambda_{\text{som}}(t)}{2}\,\bar f_{+,j}(\bar t,0)\left(1-\frac{\bar \rho_j(t,0)}{2}\right), \\
    -\bar J_{-,j}(t,0)  &= \kappa_{\beta_{-,j}}\left(1-\frac{\bar \Lambda_{\text{som}}(t)}{2}\right)\,\bar f_{-,j}(t,0),\\
    \bar J_{+,j}(t, L_j(t)) - \bar L'_j(t) \bar f_{+,j}(t, \bar L_j(t)) &= \kappa_{\beta_{+,j}}\left(1-\frac{\bar \Lambda_{j}(t)}{2}\right)\,\bar f_{+,j}(\bar t, \bar L_j(\bar t)), \\ 
    -\bar J_{-,j}(t, L_j(t))  +  \bar L'_j(t) \bar f_{-,j}(t, \bar L_j(t))&= \kappa_{\alpha_{-,j}}\frac{\bar\Lambda_j(t)}{2}\,\bar f_{-,j}(\bar t,0)\left(1-\frac{\bar \rho_j(t,0)}{2}\right),
    \end{aligned}
    \end{equation*}
    with $\kappa_{\alpha_{+,j}} = \frac{\tilde t}{\tilde L} c_{\alpha_{+,j}},  \kappa_{\alpha_{-,j}} = \frac{\tilde t}{\tilde L} c_{\alpha_{-,j}},  \kappa_{\beta_{+,j}} = \frac{\tilde t}{\tilde L} c_{\beta_{+,j}},  \kappa_{\beta_{-,j}} = \frac{\tilde t}{\tilde L} c_{\beta_{-,j}}$.
It remains to fix the values of the constants appearing in the functions $\alpha_{\pm,j},\,\beta_{\pm,j}$. As they correspond to velocities, we fix them to the typical in-/outflux velocity
$$
    \tilde c := c_{\alpha_{+,j}} = c_{\alpha_{-,j}} = c_{\beta_{+,j}}= c_{\beta_{-,j}} = 0.4 ~\frac{\mu\text{m}}{\text{s}}.
$$
For the soma and the growth cones   we choose half of the maximal amount of vesicles as typical values, i.e., $\tilde \Lambda_{\text{som}} = 3000\,\text{vesicles}$, $\tilde \Lambda_{j} = 200\,\text{vesicles}$, $j=1,2$. 
We obtain 
   \begin{equation*}
    \begin{aligned}
  &  \bar \Lambda'_{\text{som}}(t)= \kappa_{\text{som}}\sum_{j=1,2}  \left[\left(1-\frac{\bar \Lambda_{\text{som}}(t)}{2}\right)\,\bar f_{-,j}(t,0) -  \frac{\bar \Lambda_{\text{som}}(t)}{2}\,\bar f_{+,j}(\bar t,0)\left(1-\frac{\bar \rho_j(t,0)}{2}\right) \right] \\
    & \qquad \qquad \qquad + \kappa_\gamma\bar \gamma_{\text{prod}}(t) , \\ 
     &\bar \Lambda'_{j}(t) = \kappa_{\text{cone}}\left[\left(1-\frac{\bar \Lambda_{j}(t)}{2}\right)\,\bar f_{+,j}(\bar t, \bar L_j(\bar t))- \frac{\bar\Lambda_j(t)}{2}\,\bar f_{-,j}(\bar t, \bar L_j(t))\left(1-\frac{\bar \rho_j(\bar t, \bar L_j(t))}{2}\right)\right] \\
     & \qquad \qquad \qquad -\kappa_{h} \bar h_j(\bar\Lambda_j(t), \bar L_j(t)),
    \end{aligned}
    \end{equation*}
with $
    \kappa_{\text{som}} = \frac{\tilde t}{\tilde \Lambda_{\text{som}}}\tilde  c \tilde f,  \kappa_\gamma = \tilde \gamma \frac{\tilde t}{\tilde \Lambda_{\text{som}}},  \kappa_{\text{cone}} = \frac{\tilde t}{\tilde \Lambda_{\text{cone}}}\tilde  c \tilde f,  \kappa_{h} = \frac{\tilde t}{\tilde \Lambda_{\text{cone}}} \tilde h \chi. $
    Finally, for the scaled production function $\bar h$ in  \eqref{free-boundary} we make the choice
	$\bar h_j(\bar \Lambda, \bar L) = \operatorname{atan}(\bar \Lambda - \bar \Lambda_{\text{growth}})H(\bar L - \bar L_{\text{min}})$, $j=1,2$,
where $H$ is a smoothed Heaviside function. 
We have
\begin{align*}
\bar L_j'(t) = \kappa_{L}\, \bar h_j(\bar\Lambda_j(t), \bar L_j(t)), \; \text{ with } \kappa_L = \frac{\tilde t}{\tilde L}\,\tilde h.
\end{align*}

\section{Numerical studies}
\label{sect6}
We present two examples that demonstrate the capability of our model to reproduce observations in biological systems. Both start with an initial length difference of the two neurites. The first example shows that the shorter neurite can become the longer one due to a local advantage of the number of vesicles present in the growth cone while the second showcases oscillatory behaviour in neurite lengths that is observed experimentally.
Both simulations are performed in \textsc{Matlab}, using the finite volume scheme introduced in Section \ref{sec:FV}, using the parameters $\eta = 10$, $n_e = 100$, $\tau = 1\mathrm{e}{-4}$. We chose $T = 9$ (corresponding to $18$ hours) as a maximal time of the simulation, yet when a stationary state is reached before (measured by the criterium $\|f_{\pm,j}^{(n+1)} - f_{\pm,j}^{(n)}\|_2 \le 1\mathrm{e}{-9}$), the simulation is terminated. We also set $\lambda = 0$ in all simulations.

\subsection{Fast growth by local effects}
\label{sec:experiment1}
The first example shows that an initial length deficiency of a neurite can be overcome by a local advantage of vesicles on the growth cone. In this set-up we fixed (all scaled quantities) the following initial data: $L_1^0 = 1.1$, $L_2^0 = 0.9$, $\Lambda_{\text{soma}}^0 = 1$, $\Lambda_1^0= 0.65$,  $\Lambda_2^0= 1.15$, $\bbf_1^0 = \bbf_2^0 = (0.2,0.2)$. Furthermore, the functions 
\begin{align*}
g_\pm(f_{+},f_{-}) &= 1-\frac12\left(f_{+}+f_{-}\right),\quad \alpha_+(\Lambda) = 0.05 \,\frac{\Lambda}2,\quad \alpha_-(\Lambda) = 0.1\, \frac{\Lambda}2,\\ \beta_{\pm}(\Lambda) &= 0.7 \, \left(1-\frac{\Lambda}2\right),\quad h(\Lambda, L) = \frac{\tan^{-1}(\Lambda - \Lambda_{\textup{growth}})}{1+\exp(-4\,(L - L_{\textup{min}}-0.2))},\quad v_0=0.1
\end{align*} 
were chosen in this example. {The choices of $\alpha_+$ and $\alpha_-$ are motivated as follows: both are proportional to $\Lambda / 2$ as the typical values within the soma and growth cones, respectively, are chosen to be half the maximum which means that $\Lambda / 2 = 1$ if this value is reached. Thus, relative to their capacity, the outflow rates from soma and growth cone behave similarly. Now, since we are interested in the effect of vesicle concentration in the respective growth cones, we chose a small constant in $\alpha_+$ in order to limit the influence of new vesicles entering from the soma, relative to $\alpha_-$. This is a purely heuristic choice to examine if such a local effect can be observed in our model at all.}
Figure \ref{fig:Ex1_Snapshots} shows snapshots of the simulation at different times while Figure \ref{fig:Ex1_Lengths_Concentrations_small_alpha} shows the evolution of neurite lengths and vesicle concentrations over time.
The results demonstrate that the local advantage of a higher vesicle concentration in the growth cone of the shorter neurite is sufficient to outgrow the competing neurite. Yet, this requires a weak coupling in the sense that the outflow rate at the soma is small, see the constant $0.05$ in $\alpha_+$. Increasing this value, the local effect does not prevail and indeed, the longer neurite will always stay longer while both neurites grow at a similar rate as shown in Figure \ref{fig:Ex1_Lengths_Concentrations_large_alpha}. Thus we consider this result as biologically not very realistic, in particular as it cannot reproduce cycles of extension and retraction that are observed in experiments.
\begin{figure}[htb]
    \centering
    \includegraphics[width=.99\textwidth]{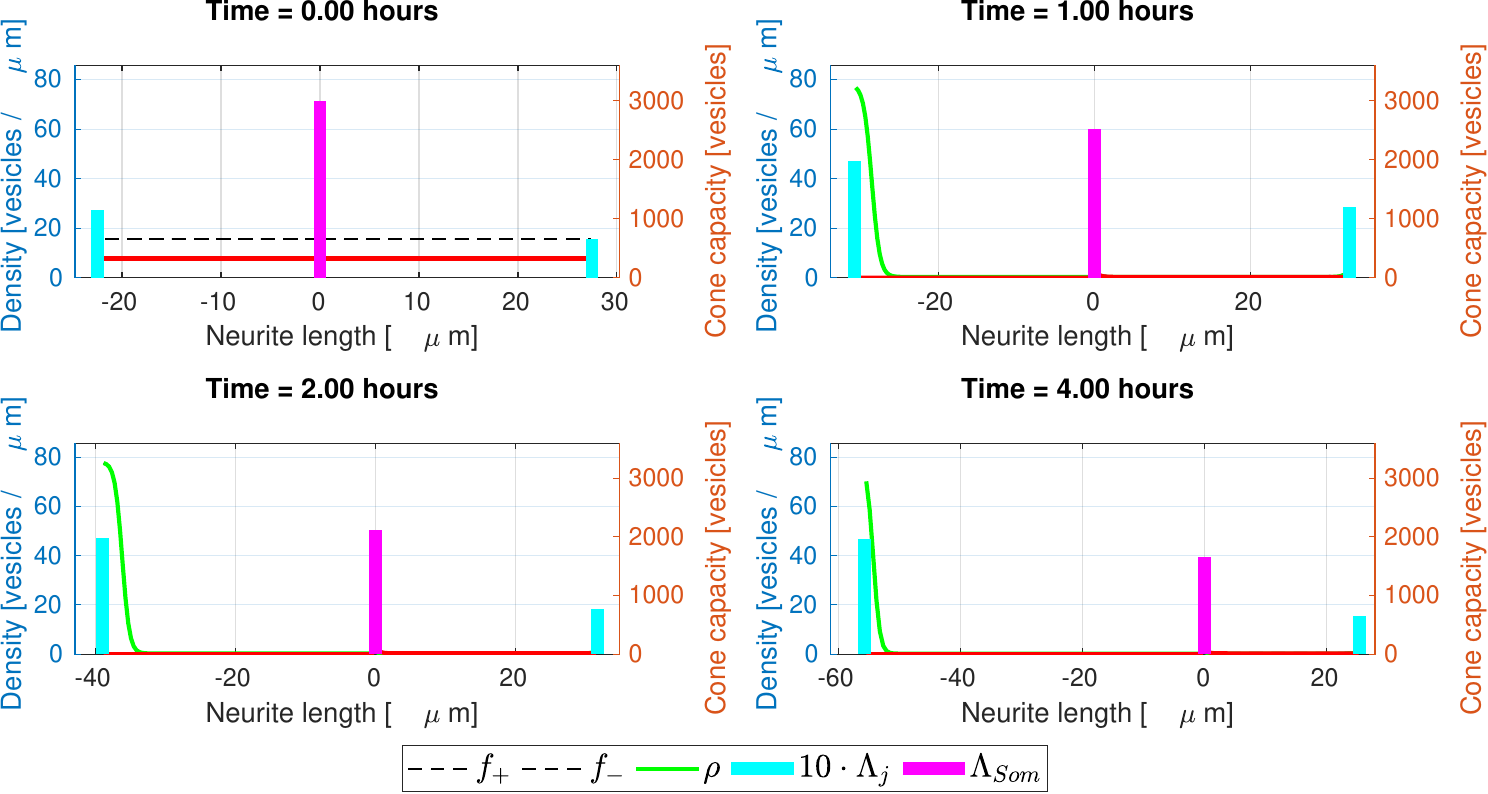}
    \caption{The vesicle densities $f_{\pm, j}$, $j=1,2$, and pool capacities $\Lambda_k$, $k\in\{\text{som},1,2\}$,  for the example from Section \ref{sec:experiment1} plotted at different time points.}
    \label{fig:Ex1_Snapshots}
\end{figure}
\begin{figure}[htb]
    \centering
    \subfloat[$\alpha_+(\Lambda) = 0.05\,\frac{\Lambda}2$\label{fig:Ex1_Lengths_Concentrations_small_alpha}]{
    \includegraphics[width=0.8\textwidth]{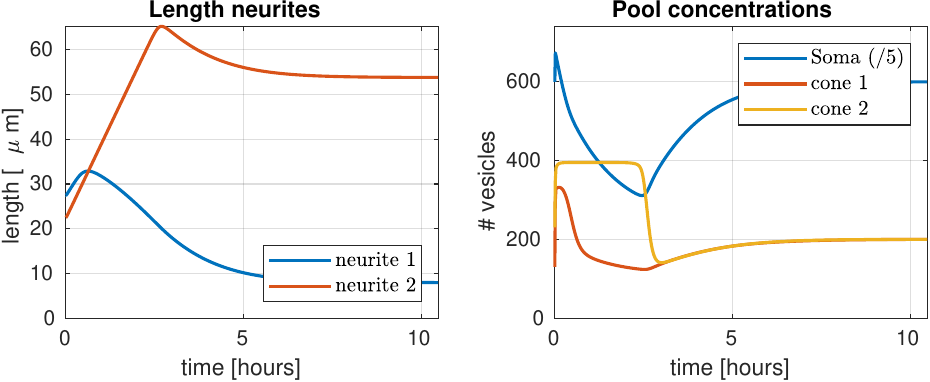}
    }\\
    \subfloat[$\alpha_+(\Lambda) = \frac{\Lambda}2$\label{fig:Ex1_Lengths_Concentrations_large_alpha}]{
    \includegraphics[width=0.8\textwidth]{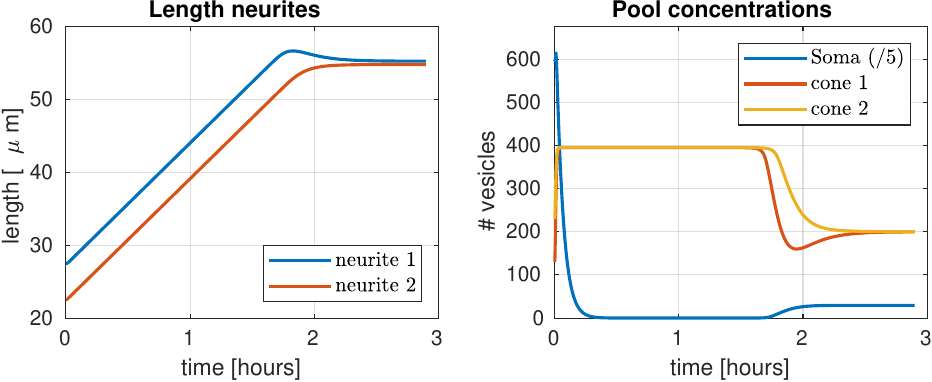}
    }
    \caption{The neurite lengths $L_j$, $j=1,2$, and pool capacities $\Lambda_k$, $k\in\{\text{som},1,2\}$, for the example from Section \ref{sec:experiment1} plotted over time.}
    \label{fig:Ex1_Lengths_Concentrations}
\end{figure}
%

\subsection{Oscillatory behaviour due to coupling of soma outflow rates to density of retrograde vesicles}
\label{sec:experiment2}
In order to overcome the purely local nature of the effect in the previous example, it seems reasonable to include effects that couple the behaviour at the growth cones to that of the soma via the concentrations of vesicles in the neurites. We propose the following two mechanisms: first, we assume that a strongly growing neurite is less likely to emit a large number of retrograde vesicles as it wants to use all vesicles for the growth process. In addition, we assume that the soma aims to reinforce strong growth and is doing so by measuring the density of arriving retrograde vesicles. The lower it becomes, the more anterograde vesicles are released. Such behaviour can easily be incorporated in our model by choosing
\begin{align*}
g_+(f_+, f_-) &= \left(\sqrt{\max(0,1-3\,f_-)^2 + 0.1}+0.5\right)\,\left(1-\frac12\,(f_++f_-)\right),\\
\alpha_+(\Lambda) &= 0.6\frac{\Lambda}2,\quad \alpha_-(\Lambda)=(1-\frac\Lambda2)\,\frac\Lambda2,\quad v_0=0.04.
\end{align*}
The remaining functions are defined as in Section~\ref{sec:experiment1}.
The initial data in this example are $L_1^0 = 1.1$, $L_2^0 = 1$, $\Lambda_{\textup{som}}^0 = 1$, $\Lambda_{1}^0=\Lambda_2^0 = 0.9$, $f_{\pm, 1}^0 = f_{\pm,2}^0 = 0$.

The results are presented in Figures \ref{fig:Ex3_Snapshots} (snapshots) and \ref{fig:Ex3_Lengths_Concentrations} and are rather interesting: first, it is again demonstrated that the shorter neurite may outgrow the larger one. Furthermore, as a consequence of the non-local coupling mechanism, {the model is able to reproduce the oscillatory cycles of retraction and growth that are sometimes observed, see e.\,g.\ \cite{Cooper2013_MechanismsCellMigration,Winans2016_ActinWaves}. Also the typical oscillation period of 2-4 hours observed in \cite[Figure 1]{Winans2016_ActinWaves} can be confirmed in our computation.} Finally, the model predicts one neurite being substantially longer than the other which one might interpret as axon and dendrite.
\begin{figure}[htb]
    \centering
    \includegraphics[width=.99\textwidth]{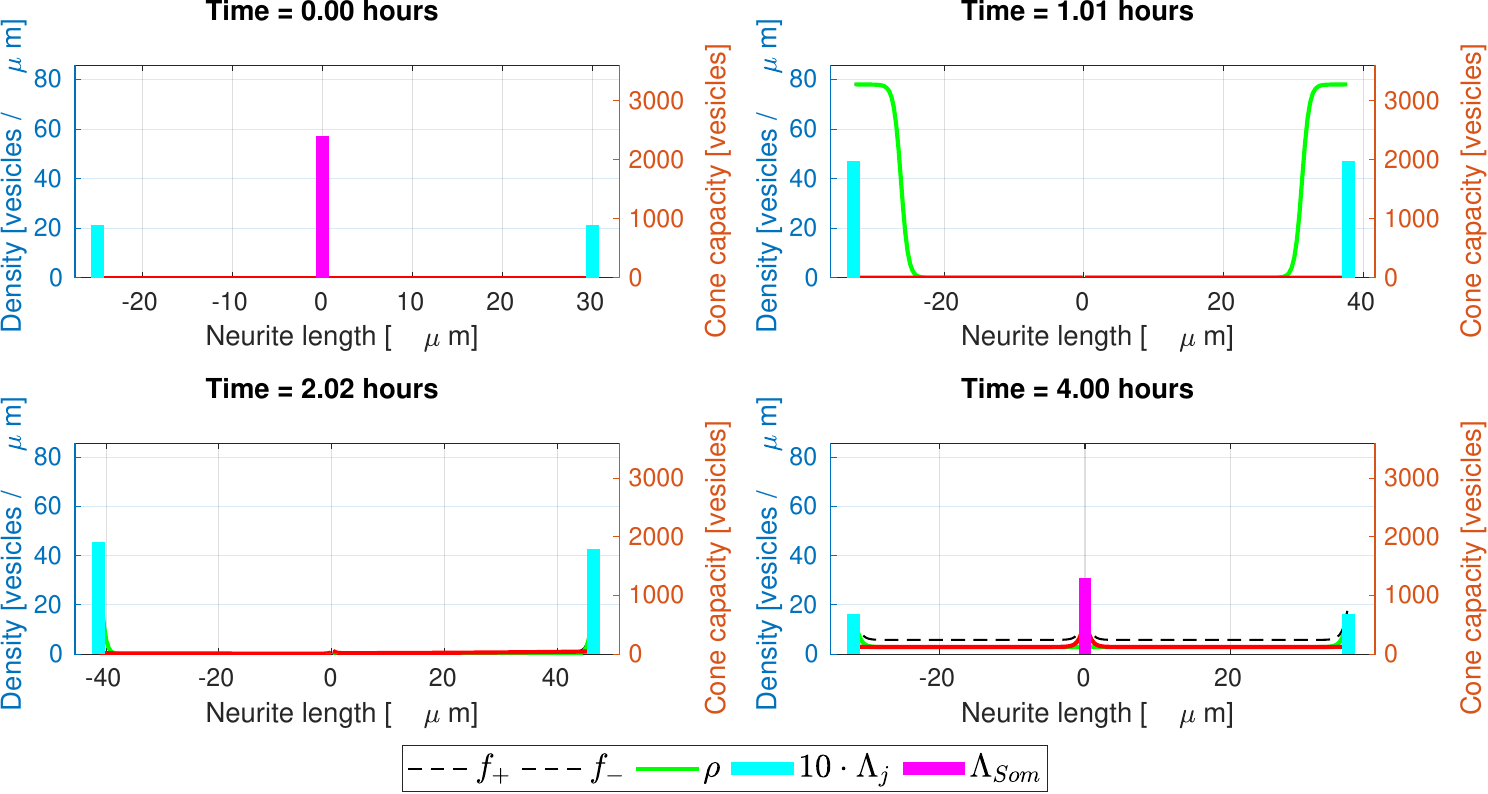}
    \caption{The vesicle densities $f_{\pm,j}$, $j=1,2$, and pool capacities $\Lambda_k$, $k\in\{\text{som},1,2\}$, for the example from Section \ref{sec:experiment2} plotted at different time points.}
    \label{fig:Ex3_Snapshots}
\end{figure}
\begin{figure}[htb]
    \centering
    \includegraphics[width=0.8\textwidth]{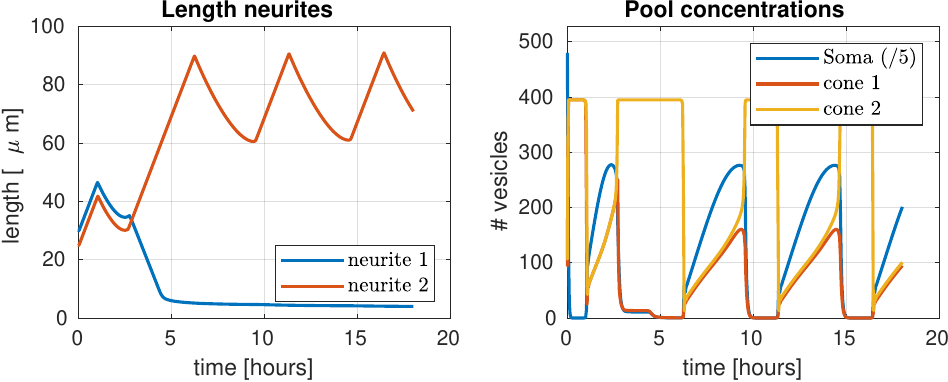}
    \caption{The neurite lengths $L_j$, $j=1,2$, and pool capacities $\Lambda_k$, $k\in\{\text{soma},1,2\}$, for the example from Section \ref{sec:experiment2} plotted over time.}
    \label{fig:Ex3_Lengths_Concentrations}
\end{figure}

\section{Conclusion \& Outlook}\label{sect7}
We have introduced a free boundary model for the dynamics of developing neurons based on a coupled system of partial- and ordinary differential equations. We provided an existence and uniqueness result for weak solutions and also presented a finite volume scheme in order to simulate the model. {Our results show that the model is able to reproduce behaviour such as retraction-elongation cycles on scales comparable to those observed in experimental measurements as shown in Section \ref{sec:experiment2}. On the other hand, the numerical results show that the density of vesicles within the neurites is, for most of the time, rather small. Thus, the effect of the non-linear transport term that we added may be questioned and indeed, rerunning the simulations with linear transport yields rather similar results. It remains to be analysed if such low vesicle densities are biologically reasonable and thus  offer the opportunity to simplify the model.}
%

A further natural question that arises at this point is what can be learned form these results. We think that while the transport mechanisms within the neurites as well as the growth and shrinkage are reasonable and fixed (up to the discussion about linear vs. non-linear transport above), most of the behaviour of the model is encoded in the coupling via the boundary conditions. These, on the other hand, allow for a large variety of choice out of which it will be difficult to decide which is the one actually implemented in a real cell. Thus, as a next step for future work, we propose to consider these couplings as unknown parameters that need to be learned using experimental data that come from experiments. We are confident that this will allow to identify possible interactions between soma and growth cones and will give new insight into the actual mechanisms at work.

Finally we remark that as our  model only focuses on the role of vesicle transport, many other effects are neglected and clearly our approach is nowhere near a complete description of the process of neurite outgrowth. To this end, we plan to extend our model further in the future, adding effects  such as the role of microtubule assembly as well as chemical signals, which are neglected so far.

\section*{Acknowledgments}
{\small GM acknowledges the support of DFG via grant MA 10100/1-1 project number 496629752, JFP via grant HE 6077/16-1 (eBer-22-57443).  JFP thanks the DFG for support via the Research Unit FOR 5387 POPULAR, Project No. 461909888. We thank Andreas P\"uschel (WWU M\"unster) for valuable discussions on the biological background. In addition, we are very grateful to the comments of the anonymous referees.}

\section*{Competing interests} The author(s) declare none.

\section*{Appendix}

For convenience of the reader we state \cite[Theorem 5.3]{Deimling} about invariant regions of solutions of ODEs.
\begin{theorem}\label{thm:invariant_deimling}
Let $X$ be a real normed linear space, $\Omega \subset X$ an open set and $D \subset X$ a distance set with $D \cap \Omega \neq \emptyset$. Let $f:(0, a) \times \Omega \rightarrow X$ be such that
\begin{itemize}
\item[(A1)]
$$
\begin{aligned}
& (f(t, x)-f(t, y), x-y)_{+} \leq \omega(t,|x-y|)|x-y| \\
& \quad \text { for } x \in \Omega \backslash D, y \in \Omega \cap \partial D, t \in(0, a),
\end{aligned}
$$
where $\omega:(0, a) \times \mathbb{R}^{+} \rightarrow \mathbb{R}$ is such that $p(t) \leq 0$ in $(0, \tau) \subset(0, a)$ whenever $\rho:[0, \tau) \rightarrow \mathbb{R}^{+}$is continuous, $\rho(0)=0$ and $D^{+} \rho(t) \leq \omega(t, \rho(t))$ for every $t \in(0, \tau)$ with $\rho(t)>0$ (where $D^+$ denotes the one-sided derivative with respect to $t$).
\item[(A2)] If $x \in \Omega \cap \partial D$ is such that the set of outward normal vectors $N(x)$ is non-empty and 
$$
(f(t, x), \nu)_{+} \leq 0
$$
for all $\nu \in N(x)$ and $t \in(0, a)$.
\end{itemize}
Then $D \cap \Omega$ is forward invariant with respect to $f$, i.e., any continuous $x:[0, b) \rightarrow \Omega$, such that $x(0) \in D$ and $x^{\prime}=f(t, x)$ in $(0, b)$, satisfies $x(t) \in D$ in $[0, b)$.
\end{theorem}
\bibliographystyle{siamplain}
\bibliography{references}
\end{document}

%% file: ex_shared.tex
\usepackage{amsfonts}
\usepackage{graphicx}
\usepackage{epstopdf}
\usepackage{subfig}
\usepackage{algorithmic}
\usepackage{stmaryrd}
\ifpdf
  \DeclareGraphicsExtensions{.eps,.pdf,.png,.jpg}
\else
  \DeclareGraphicsExtensions{.eps}
\fi

\DeclareMathOperator*{\esssup}{ess ~sup}         %
\newcommand{\Om}{0,L}

\newcommand{\R}{\mathbb{R}}

\newcommand{\dt}{\,\textup{d}t}
\newcommand{\ds}{\,\textup{d}s}
\newcommand{\dx}{\,\textup{d}x}
\newcommand{\dy}{\,\textup{d}y}

\newcommand{\eps}{\varepsilon}

\renewcommand{\l}{\left}
\renewcommand{\r}{\right}
\numberwithin{equation}{section}
\renewcommand{\theequation}{\arabic{section}.\arabic{equation}}

\newcommand{\bbf}{\boldsymbol{f}}

\newcommand\avg[1]{\{\!\!\{#1\}\!\!\}}
\newcommand\jump[1]{\llbracket#1\rrbracket}

\makeatletter
\providecommand{\cref@override@label@type}[1]{}
\makeatother

\newsiamremark{remark}{Remark}
\newsiamremark{hypothesis}{Hypothesis}
\crefname{hypothesis}{Hypothesis}{Hypotheses}
\newsiamthm{claim}{Claim}

\headers{A free boundary model for transport induced neurite growth}{G. Marino, J.-F. Pietschmann, and M. Winkler}

\title{A free boundary model for transport induced neurite growth}

\author{Greta Marino$^{\dagger,}$\thanks{ Universit\"{a}t Augsburg, Institut f\"ur Mathematik, Universit\"{a}tsstra\ss e 12a, 86159 Augsburg, Germany
  (\email{greta.marino@uni-a.de, jan-f.pietschmann@uni-a.de}).}
\and Jan-Frederik Pietschmann$^{*,}$\thanks{Centre for Advanced Analytics and Predictive Sciences (CAAPS), University of Augsburg,
Universit\"{a}tsstr. 12a, 86159 Augsburg, Germany.}
\and Max Winkler\thanks{Technische Universit\"{a}t Chemnitz, Fakult\"at f\"ur Mathematik, Reichenhainer Stra\ss e 41, 09126 Chemnitz, Germany
  (\email{max.winkler@mathematik.tu-chemnitz.de}).}}

\usepackage{amsopn}



%% file: Sketch-Neurons2ohneBoundary.pdf_tex
\begingroup%
  \makeatletter%
  \providecommand\color[2][]{%
    \errmessage{(Inkscape) Color is used for the text in Inkscape, but the package 'color.sty' is not loaded}%
    \renewcommand\color[2][]{}%
  }%
  \providecommand\transparent[1]{%
    \errmessage{(Inkscape) Transparency is used (non-zero) for the text in Inkscape, but the package 'transparent.sty' is not loaded}%
    \renewcommand\transparent[1]{}%
  }%
  \providecommand\rotatebox[2]{#2}%
  \newcommand*\fsize{\dimexpr\f@size pt\relax}%
  \newcommand*\lineheight[1]{\fontsize{\fsize}{#1\fsize}\selectfont}%
  \ifx\svgwidth\undefined%
    \setlength{\unitlength}{371.49041161bp}%
    \ifx\svgscale\undefined%
      \relax%
    \else%
      \setlength{\unitlength}{\unitlength * \real{\svgscale}}%
    \fi%
  \else%
    \setlength{\unitlength}{\svgwidth}%
  \fi%
  \global\let\svgwidth\undefined%
  \global\let\svgscale\undefined%
  \makeatother%
  \begin{picture}(1,0.18225515)%
    \lineheight{1}%
    \setlength\tabcolsep{0pt}%
    \put(0,0){\includegraphics[width=\unitlength,page=1]{Sketch-Neurons2ohneBoundary.pdf}}%
    \put(0.45240993,-2.42404239){\color[rgb]{0,0,0}\makebox(0,0)[lt]{\begin{minipage}{0.39242402\unitlength}\raggedright \end{minipage}}}%
    \put(0,0){\includegraphics[width=\unitlength,page=1]{Sketch-Neurons2ohneBoundary.pdf}}%
    \put(0.03103472,0.08226451){\color[rgb]{0,0,0}\makebox(0,0)[lt]{\lineheight{1.25}\smash{\begin{tabular}[t]{l}$\textcolor{white}{\Lambda_{1}}$\end{tabular}}}}%
    \put(0.43855079,0.0822645){\color[rgb]{0,0,0}\makebox(0,0)[lt]{\lineheight{1.25}\smash{\begin{tabular}[t]{l}$\textcolor{white}{\Lambda_{\text{som}}}$\end{tabular}}}}%
    \put(0.8706924,0.08226451){\color[rgb]{0,0,0}\makebox(0,0)[lt]{\lineheight{1.25}\smash{\begin{tabular}[t]{l}$\textcolor{white}{\Lambda_{2}}$\end{tabular}}}}%
    \put(0.37451735,0.05193472){\color[rgb]{0,0,0}\makebox(0,0)[lt]{\lineheight{1.25}\smash{\begin{tabular}[t]{l}$0$\end{tabular}}}}%
    \put(0.53929796,0.05193472){\color[rgb]{0,0,0}\makebox(0,0)[lt]{\lineheight{1.25}\smash{\begin{tabular}[t]{l}$0$\end{tabular}}}}%
    \put(0.09744847,0.0543269){\color[rgb]{0,0,0}\makebox(0,0)[lt]{\lineheight{1.25}\smash{\begin{tabular}[t]{l}$L_1(t)$\end{tabular}}}}%
    \put(0.77549625,0.0543269){\color[rgb]{0,0,0}\makebox(0,0)[lt]{\lineheight{1.25}\smash{\begin{tabular}[t]{l}$L_2(t)$\end{tabular}}}}%
    \put(0.36393016,0.16180333){\color[rgb]{0,0,0}\makebox(0,0)[lt]{\lineheight{1.25}\smash{\begin{tabular}[t]{l}$\beta_{-,1}$\end{tabular}}}}%
    \put(0.50401385,0.16180333){\color[rgb]{0,0,0}\makebox(0,0)[lt]{\lineheight{1.25}\smash{\begin{tabular}[t]{l}$\beta_{-,2}$\end{tabular}}}}%
    \put(0.50516746,0.00634848){\color[rgb]{0,0,0}\makebox(0,0)[lt]{\lineheight{1.25}\smash{\begin{tabular}[t]{l}$\alpha_{+,2}$\end{tabular}}}}%
    \put(0.36211446,0.00634848){\color[rgb]{0,0,0}\makebox(0,0)[lt]{\lineheight{1.25}\smash{\begin{tabular}[t]{l}$\alpha_{+,1}$\end{tabular}}}}%
    \put(0.06634633,0.00634848){\color[rgb]{0,0,0}\makebox(0,0)[lt]{\lineheight{1.25}\smash{\begin{tabular}[t]{l}$\beta_{+,1}$\end{tabular}}}}%
    \put(0.80186955,0.00634848){\color[rgb]{0,0,0}\makebox(0,0)[lt]{\lineheight{1.25}\smash{\begin{tabular}[t]{l}$\beta_{+,2}$\end{tabular}}}}%
    \put(0.06998232,0.16180333){\color[rgb]{0,0,0}\makebox(0,0)[lt]{\lineheight{1.25}\smash{\begin{tabular}[t]{l}$\alpha_{-,1}$\end{tabular}}}}%
    \put(0.80118621,0.16180333){\color[rgb]{0,0,0}\makebox(0,0)[lt]{\lineheight{1.25}\smash{\begin{tabular}[t]{l}$\alpha_{-,2}$\end{tabular}}}}%
  \end{picture}%
\endgroup%